\documentclass[12pt,a4paper]{amsart}
\usepackage[marginratio=1:1]{geometry}
\date{\today}

\usepackage[T1]{fontenc}
\usepackage{ae,aecompl}
\usepackage[utf8]{inputenc}
\usepackage[american]{babel}
\usepackage[draft=false]{hyperref}
\usepackage{xcolor}
\hypersetup{
	colorlinks,
	linkcolor={red!50!black},
	citecolor={blue!50!black},
	urlcolor={blue!80!black}
}
\usepackage{enumitem}
\usepackage{graphicx}
\usepackage{tikz}
\usetikzlibrary{arrows.meta}
\usetikzlibrary{matrix, arrows}

\makeatletter
\tikzset{
	edge node/.code={%
		\expandafter\def\expandafter\tikz@tonodes\expandafter{\tikz@tonodes #1}}}
\makeatother
\tikzset{
	subseteq/.style={
		draw=none,
		edge node={node [sloped, allow upside down, auto=false]{$\subseteq$}}},
	Subseteq/.style={
		draw=none,
		every to/.append style={
			edge node={node [sloped, allow upside down, auto=false]{$\subseteq$}}}
	}
}

\usepackage[backend=biber,
	maxnames=99,
	minalphanames=4,
	maxalphanames=4,
	isbn=true,
	backref=true,
	citestyle=alphabetic,
	bibstyle=alphabetic,
	autocite=inline,
	sorting=nty,]{biblatex}

\usepackage{amsfonts}
\usepackage{amssymb}
\usepackage{tikz-cd}
\usepackage{float}

\newcommand{\fC}{{\mathfrak C}}
\newcommand{\cM}{{\mathcal M}}

\newcommand{\bbN}{{\mathbb{N}}}

\newcommand{\bbR}{{\mathbb{R}}}
\newcommand{\bbZ}{{\mathbb{Z}}}
\newcommand{\bbQ}{{\mathbb{Q}}}

\newcommand{\cB}{{\mathcal B}}
\newcommand{\restr}{\mathord{\upharpoonright}}
\newcommand{\EZ}{\mathrel{ { {\mathbb E}_0 } } }
\newcommand{\Er}{\mathrel{E}}

\newcommand{\proves}{\vdash}

\DeclareMathOperator{\tp}{{tp}}

\DeclareMathOperator{\dcl}{{dcl}}
\DeclareMathOperator{\Th}{{Th}}

\DeclareMathOperator{\Gal}{{Gal}}

\DeclareMathOperator{\cl}{{cl}}

\DeclareMathOperator{\id}{{id}}
\DeclareMathOperator{\Aut}{{Aut}}
\DeclareMathOperator{\Homeo}{{Homeo}}
\DeclareMathOperator{\Autf}{{Aut\,f}}

\DeclareMathOperator{\Core}{{Core}}

\DeclareMathOperator{\st}{{st}}

\newtheorem*{mainthm}{Main Theorem}
\newtheorem{thm}{Theorem}[section]

\newtheorem{ques}[thm]{Question}

\newtheorem{lem}[thm]{Lemma}
\newtheorem{fct}[thm]{Fact}
\newtheorem{cor}[thm]{Corollary}
\newtheorem{prop}[thm]{Proposition}

\theoremstyle{remark}
\newtheorem{rem}[thm]{Remark}
\theoremstyle{definition}
\newtheorem{dfn}[thm]{Definition}

\newtheorem*{clm*}{Claim}
\newtheorem{ex}[thm]{Example}
\newcounter{claimcounter}[thm]
\newenvironment{clm}{\stepcounter{claimcounter}{\noindent {\textbf{Claim}} \theclaimcounter:}}{}
\newenvironment{clmproof}[1][\proofname]{\proof[#1]}{\endproof}

\newcommand{\xqed}[1]{%
	\leavevmode\unskip\penalty9999 \hbox{}\nobreak\hfill
	\quad\hbox{\ensuremath{#1}}}

\author{Krzysztof Krupi\'nski}
\email[K.\ Krupi\'{n}ski]{kkrup@math.uni.wroc.pl}
\thanks{The first author is supported by the Narodowe Centrum Nauki grants no. 2015/19/B/ST1/01151 and 2016/22/E/ST1/00450}

\author{Tomasz Rzepecki}
\email[T.\ Rzepecki]{tomasz.rzepecki@math.uni.wroc.pl}
\address[K.\ Krupiński, T.\ Rzepecki]{
	Instytut Matematyczny, Uniwersytet Wrocławski\\
	pl. Grunwaldzki 2/4\\
	50-384 Wrocław, Poland
}
\thanks{The second author is supported by the Narodowe Centrum Nauki, via grant	 no. 2015/17/N/ST1/02322 and the doctoral scholarship 2017/24/T/ST1/00224} 
\title{Galois groups as quotients of Polish groups}
\date{\today}



\begin{document}
	\begin{abstract}
		We present the (Lascar) Galois group of {\em any} countable theory as a quotient of a compact Polish group by an $F_\sigma$ normal subgroup: in general, as a topological group, and under NIP, also in terms of Borel cardinality. This allows us to obtain similar results for arbitrary strong types defined on a single complete type over $\emptyset$. As an easy conclusion of our main theorem, we get the main result of \cite{KPR15} which says that for any strong type defined on a single complete type over $\emptyset$, smoothness is equivalent to type-definability.

		We also explain how similar results are obtained in the case of bounded quotients of type-definable groups. This gives us a generalization of a former result from \cite{KPR15} about bounded quotients of type-definable subgroups of definable groups.
	\end{abstract}
	\maketitle
	
	\section{Introduction}
	The (Lascar) Galois group of a first order theory (see Definition \ref{definition: Galois groups}) is a model-theoretic invariant, generalizing the notion of the absolute Galois group from field theory. The study of the Galois group is closely tied to the so-called strong types (see Definition \ref{definition: strong types}), which are highly relevant for generalizations of stability theory, and to model-theoretic connected group components, which for example were essential in Pillay's conjecture.

	For countable stable theories (e.g. algebraically closed fields), and, more generally, for countable G-compact theories, the Galois group is a compact Polish group. For arbitrary theories, it is still a compact topological group, but it need not be Hausdorff. So a general question is how to view Galois groups and spaces of strong types as mathematical (topological) objects and how to measure their complexity. In \cite{KPS13} it was proposed to do it via the descriptive set theoretic notion of Borel cardinality. Some deep results in this direction were obtained in \cite{KMS14, KM14, KR16}. A completely new approach via topological dynamics was developed in \cite{KP17, KPR15}. In particular, in \cite{KPR15}, it was proved that the descriptive set theoretic smoothness of a strong type defined on a single complete type over $\emptyset$ is equivalent to its type-definability. The key idea was to present the Galois group as a quotient of a compact Hausdorff group, which is interesting in its own right. However, even if the underlying theory is countable, the compact group obtained in \cite{KPR15} is not in general Polish (equivalently, metrizable), which is a serious obstacle if one wants to use it to compute Borel cardinalities of Galois groups or strong types.
	
	In this paper, we use topological dynamics for automorphism groups of suitably chosen countable models, based on the one developed in \cite{KPR15} for automorphisms of the monster model, to show that in a very strong sense (preserving much of the relation to strong type spaces, enough to estimate the Borel cardinality), the Galois group of an arbitrary countable theory is actually a quotient of a compact Polish group. We also get a similar result for any strong type defined on a single complete type over $\emptyset$.
	
	\begin{mainthm}
		\phantomsection
		\label{mainthm}
		The Galois group of a countable first order theory is the quotient of a compact Polish group by an $F_\sigma$ normal subgroup. The space of classes of a bounded invariant equivalence relation $E$ defined on single complete type over $\emptyset$ (in a countable theory) is the quotient of a compact Polish group by some subgroup (which inherits the good descriptive set theoretic properties of $E$).
	\end{mainthm}
	For the precise statement of the conclusion, see Theorem~\ref{thm:main} and Corollary \ref{cor:main theorem for Gal}. (Note that the conclusion is stronger under NIP.) For related statements, see also Theorems~\ref{thm:main_smaller} and \ref{thm:main_group}.

	\subsection*{Related work}
	As already mentioned, the present work was pre-empted by the consideration of Borel cardinalities, and, in particular, (non-)smoothness of strong types and its relation to type-definability.
	
	The equivalence of smoothness and type-definability was first conjectured for the Lascar strong type in \cite[Conjecture 1]{KPS13}, and then proved in \cite[Main Theorem A]{KMS14}. The direction was subsequently explored in \cite{KM14} and \cite{KR16} via descriptive-set-theoretical tools (related to those used in \cite{KMS14}), extending the equivalence to the so-called orbital $F_\sigma$ strong types.
	
	In \cite{KPR15}, a much more general result was obtained by using completely different methods (including substantial use of topological dynamics).
	\begin{fct}
		\label{fct:main_KPR}
		Assume that the language is countable. Let $E$ be a Borel (or, more generally, analytic) strong type on $p(\fC)$ for some $p\in S(\emptyset)$ (in countably many variables). Then exactly one of the following holds:
		\begin{itemize}[nosep]
			\item
			$E$ is relatively definable (on $p(\fC)$), smooth, and has finitely many classes,
			\item
			$E$ is not relatively definable, but it is type-definable, smooth, and has $2^{\aleph_0}$ classes,
			\item
			$E$ is neither type definable nor smooth, and it has $2^{\aleph_0}$ classes.
		\end{itemize}
	\end{fct}
	\begin{proof}
		This is \cite[Corollary 6.1]{KPR15}.
	\end{proof}
	(See also \cite[Corollary 4.10]{Rz16} for a generalization of Fact~\ref{fct:main_KPR} to a certain class of strong types not necessarily defined on a single $p(\fC)$.)
	
	In the proof of Fact~\ref{fct:main_KPR}, the main idea was the following: we consider the natural map $\Gal(T)\to p(\fC)/E$, find a compact Hausdorff group $G$ whose quotient is $\Gal(T)$ and such that the equivalence relation on $G$ induced from equality via the composed map $G\to p(\fC)/E$ is closed if and only if $E$ is type-definable, along with several other technical properties. The group $G$ constructed there is a priori very large (and not metrizable), so the standard notions of smoothness and a Borel equivalence relation do not work as well as we would like, and thus a weak analogue was used instead (using the Souslin operation and the class of sets with the Baire property). Furthermore, the aforementioned equivalence relation on $G$ is the coset equivalence relation of some $H\leq G$. Because of this, it was possible to use classical results related to compact topological groups (similar to Fact~\ref{fct:trich_polish} below) to derive Fact~\ref{fct:main_KPR}.
	Much of the difficulty of the proof lies in the construction of the group $G$. It is performed using topological dynamical tools for the automorphism group of the monster model, based on the ones developed in \cite{KP17} for definable groups.
	
	Broadly, we could say that the main goal of this paper is to replace the group $G$ above by a compact Polish group, and to obtain stronger restrictions on the Borel cardinalities of the Galois group and strong type spaces. Somewhat more precisely, we show that any strong type on a $p(\fC)$ is in a strong (particularly under NIP) sense equivalent to the relation of lying in the same left coset of some subgroup of a compact Polish group (Theorem~\ref{thm:main}). One can hope that this could be helpful in further study of Borel cardinalities of strong types (e.g.\ related to Conjectures 2 and 5.7 in \cite{KPS13}). In any case, we obtain an alternative (and arguably, more natural) proof of Fact~\ref{fct:main_KPR} by reducing it to the following trichotomy.
	
	\begin{fct}
		\label{fct:trich_polish}
		Suppose $G$ is a compact Polish group and $H\leq G$. Suppose in addition that $H$ has the strict Baire property in $G$, i.e. $H \cap C$ has the Baire property in $C$ for any closed $C \subseteq G$ (which is for example the case when $H$ is Borel or, more generally, analytic). Then exactly one of the following conditions holds:
		\begin{enumerate}
			\item
			$H$ is clopen in $G$, and so $[G:H]$ is finite,
			\item
			$[G:H]=2^{\aleph_0}$, $H$ is closed in $G$ (and so the relation of lying in the same left coset of $H$ is smooth in $G$),
			\item
			$[G:H]=2^{\aleph_0}$ and the relation of lying in the same left coset of $H$ is not smooth in $G$ (and hence $H$ is not closed).
		\end{enumerate}
	\end{fct}
	\begin{proof}
		By the Pettis theorem \cite[Theorem 9.9]{Kec95}, 
		%(a product of non-meager sets with the Baire property has nonempty interior), 
		if $H$ has the Baire property but is not open, it must be meager, and hence, by Mycielski's theorem \cite[Theorem 5.3.1]{SuG}, $[G:H]=2^{\aleph_0}$. On the other hand, from \cite[Theorem 1]{Mil77}, we deduce that if $G/H$ is smooth (equivalently, there is a countable family of Borel sets separating left cosets of $H$), then $H$ must be closed. The fact that closedness implies smoothness is well-known (e.g. see \cite[Corollary 1.32]{KMS14}).
	\end{proof}

	We will also briefly explain how our methods can be adapted to show a variant of Fact~\ref{fct:main_KPR} for arbitrary type-definable groups (which in \cite{KPR15} was only shown for type-definable subgroups of definable groups), in the form of Corollary \ref{cor:main_group}.

	\subsection*{Structure of the paper}

	Section \ref{section: preliminaries} contains basic definitions and facts. 
	
	In Section \ref{section: relations coarser than the Kim-Pillay strong type}, we give a simple proof of the main theorem for strong types coarser than Kim-Pillay strong type. The point is that for such strong types we do not need to develop and use the machinery related to topological dynamics. Instead, we use the Kim-Pillay Galois group, and then focus only on the descriptive set theoretic aspect of the proof. Also, the NIP assumption in the ``moreover part" of Theorem~\ref{thm:main} can be dropped for strong types coarser than KP strong type. 
	
	Section \ref{section: topological dynamics} collects various known definitions and facts from topological dynamics in the form suitable for applications in further sections. In Subsections \ref{section: Rosenthal compacta} and \ref{section: tame systems}, we take the opportunity and present more than is needed in our main applications (where only metrizable systems are used), providing in particular precise references to topological dynamics papers, which we hope will be helpful for future applications of tame systems in model theory and can serve as a reference.
	
	In Section \ref{section: independence, tameness and ambition}, we recall and slightly develop the correspondence between tameness in topological dynamics and NIP in model theory. We also introduce a new notion of an ambitious model, which is essential in the main theorem. 
	
	Section \ref{sec:top_dyn_to_Polish} is essentially new. It contains a general topological dynamical development, the main outcome of which is a construction of a Polish compact group associated with a given metrizable dynamical system.
	
	In Subsection \ref{subsection: adaptation from KPR15}, we adapt the dynamics developed in \cite{KPR15} for the group of automorphisms of a monster model to the dynamics of the groups of automorphisms of countable ambitious models. Finally, in Subsection \ref{subsection: proof of the main theorem}, we give a proof of our main theorem (i.e. Theorem ~\ref{thm:main}), using the theory developed in Sections \ref{sec:top_dyn_to_Polish} and \ref{subsection: adaptation from KPR15}.
	
	In Subsection \ref{subsection: smaller domains}, we extend the context of Theorem~\ref{thm:main} to strong types restricted to appropriate type-definable subsets of the domain. In Subsection \ref{subsection: type-definable groups}, we briefly explain how the method of the proof of Theorem~\ref{thm:main} adapts to show a variant this theorem for bounded quotients of type-definable groups.
	
	In the appendix, we compute the Ellis group of the the flow $(\Aut(M), S_m(M))$ for $M$ being the unique countable models of certain non-G-compact $\omega$-categorical theories from \cite{CLPZ01} and \cite{KPS13} and $m$ being an enumeration of $M$ ($S_m(M)$ denotes the space of complete types over $M$ extending $\tp(m/\emptyset)$). Using this together with our main theorem, we compute the Galois groups and their Borel cardinalities in these examples, confirming what is claimed in \cite[Remark 5.3]{KPS13} (via different methods).

	\section{Preliminaries}\label{section: preliminaries}
	In this section, we recall basic definitions related to Borel equivalence relations in model theory. For a more detailed exposition, one can refer to \cite{KPS13}, \cite{CLPZ01}, \cite{KMS14}, \cite{KR16}, or \cite{KPR15}.

	\subsection{Topology}
	
	In this paper, compact spaces are not Hausdorff by definition, so we will add the adjective ``Hausdorff" whenever it is needed.
	
	Recall that for a compact Hausdorff space $X$ the following conditions are equivalent:
	\begin{itemize}
		\item $X$ is second countable,
		\item $X$ is is metrizable,
		\item $X$ is Polish (i.e. separable and completely metrizable).
	\end{itemize}
	
	\begin{fct}\label{fct: preservation of metrizability}
		Metrizability is preserved under continuous surjections between compact, Hausdorff spaces.
	\end{fct}
	\begin{proof}
		This follows from \cite[Theorem 4.4.15]{Eng89}.
	\end{proof}
	
	\begin{dfn}
	A surjection $f \colon X \to Y$ between topological spaces is said to be a {\em topological quotient map} if it has the property that a subset $A$ of $Y$ is open [closed] if an only if $f^{-1}[A]$ is open [closed]. This is equivalent to saying that the induced bijection $X/E \to Y$ is a homeomorphism, where $E$ in the equivalence relation of lying in the same fiber of $f$ and $X/E$ is equipped with the quotient topology.
	\xqed{\lozenge}
	\end{dfn}
	
	The next remark follows from the fact that continuous functions between compact Hausdorff spaces are closed maps.
	
	\begin{rem}\label{rem: continuous surjection is closed}
		A continuous surjection between compact Hausdorff spaces is a quotient topological map.
		\xqed{\lozenge}
	\end{rem}
	
	\begin{fct}
		\label{fct:quotient_by_closed_subgroup}
		If $G$ is a topological group (i.e.\ a group equipped with possibly non-Hausdorff topology with respect to which the group operations are continuous) and $H\leq G$, then $G/H$ (with the quotient topology) is Hausdorff if and only if $H$ is closed in $G$.
	\end{fct}
	\begin{proof}
		See \cite[III.2.5, Proposition 13]{NB66}.
	\end{proof}
	
	\subsection{Borel cardinality}
	\begin{dfn}
		Suppose $E$ and $F$ are equivalence relations on Polish spaces $X$ and $Y$. We say that \emph{$E$ is Borel reducible to $F$} --- written $E\leq_B F$ --- if there is a Borel function $f\colon X\to Y$ such that $x_1\Er x_2$ if and only if $f(x_1)\mathrel{F} f(x_2)$.
		
		If $E\leq_B F$ and $F\leq_B E$, we say that $E$ and $F$ are \emph{Borel equivalent}, written $E\sim_B F$. In this case, we also say that $E$ and $F$, or $X/E$ and $Y/F$, have the same {\em Borel cardinality}; informally speaking, the {\em Borel cardinality} of $E$ is its $\sim_B$-equivalence class.
		\xqed{\lozenge}
	\end{dfn}
	
	\begin{dfn}
		We say that an equivalence relation $E$ on a Polish space $X$ (or the quotient $X/E$) is {\em smooth} if $E$ is Borel reducible to equality on $2^\bbN$ (or, equivalently, $\bbR$).
		\xqed{\lozenge}
	\end{dfn}
	
	\begin{fct}
		\label{fct:borel_section}
		If $X,Y$ are compact Polish spaces and $f\colon X\to Y$ is a continuous surjection, then $f$ has a Borel section $g$. In particular, if $f$ is a continuous reduction from $E$ on $X$ to $F$ on $Y$, then $g$ is a Borel reduction from $F$ to $E$, hence $E\sim_B F$.
	\end{fct}
	\begin{proof}
		\cite[Exercise 24.20]{Kec95}
	\end{proof}
	
	Recall that analytic subsets of Polish spaces are closed under taking images and preimages by Borel maps \cite[Proposition 4.14]{Kec95}. Borel subsets of Polish spaces are clearly closed under preimages by Borel functions. Using Fact \ref{fct:borel_section}, one easily gets that
	whenever $f \colon X \to Y$ is a continuous surjection between compact Polish spaces, then a subset $B$ of $Y$ is Borel if and only if $f^{-1}[B]$ is Borel in $X$.

	\subsection{Model theory and notation}
	Throughout, $T$ will denote the ambient (first order, complete, often countable) theory. The arguments and results in this paper work for multi-sorted theories with straightforward modifications, but for simplicity, we assume that $T$ is single-sorted, unless specified otherwise.
	
	We fix a strong limit cardinal $\kappa$ larger than $\lvert T\rvert$ and ``all the objects we are interested in".
	
	\begin{dfn}
		A \emph{monster model} is a model $\fC$ of $T$ which is $\kappa$-saturated (i.e.\ each type over an arbitrary set of parameters from $\fC$ of size less than $\kappa$ is realized in $\fC$) and strongly $\kappa$-homogeneous (i.e.\ any elementary map between subsets of $\fC$ of cardinality less than $\kappa$ extends to an isomorphism of $\fC$).
		\xqed{\lozenge}
	\end{dfn}
	
	We fix a monster model $\fC$ and assume that all models we discuss are elementary submodels of $\fC$. (For the existence of a monster model see \cite[Theorem 10.2.1]{Hod93}.)
	
	By \emph{small} we mean smaller than our chosen $\kappa$. When we write $X\subseteq \fC$ we mean that $X$ is a subset of an arbitrary small power of $\fC$. When $a$ is a tuple in $\fC$ and $A\subseteq \fC$, by $S_a(A)$ we mean the subspace of $S(A)$ consisting of types extending the type of $a$ over $\emptyset$. When $A,X\subseteq \fC$, by $X_A$ we mean the subspace of $S(A)$ consisting of types of elements of $X$ over $A$. By $\equiv$ we denote the relation of having the same type over $\emptyset$ (equivalently, of lying in the same orbit of $\Aut(\fC)$).

	\subsection{Strong types}
	\begin{dfn}\label{definition: strong types}
		A \emph{bounded invariant equivalence relation} is an equivalence relation on an ($\Aut(\fC)$-)invariant set $X$ which is itself ($\Aut(\fC)$-)invariant (as a subset of $X^2$), and which has a small number of classes.
		
		A \emph{strong type} is a single class of a bounded invariant equivalence relation finer than $\equiv$, or, abusing the terminology, any relation of this kind.
		\xqed{\lozenge}
	\end{dfn}
	
	Now, we recall some definitions related to descriptive set theoretic treatment of strong types. For more details see \cite[Section 2.1]{KR16}.
	
	\begin{fct}\label{fact: refining by type over M}
		\label{fct:stypes_from_types}
		If $M\prec \fC$ is a small model and $E$ is a bounded invariant equivalence relation on $X$, then the $E$-classes are setwise $\Aut(\fC/M)$-invariant.
		
		Consequently, the quotient map $X\to X/E$ factors through $X\to X_M$, yielding a natural map $X_M\to X/E$.
	\end{fct}
	\begin{proof}
		This is well-known. It follows from the fact that whenever $a,b \in X$ have the same type over $M$, then there is a sequence $\bar c= (c_i)_{i < \kappa}$ in $\fC$ such that the sequences $a ^\frown \bar c$ and $b ^\frown \bar c$ are both indiscernible, because then $a \Er c_0 \Er b$ (as otherwise, by indiscernibility, $E$ would have at least $\kappa$ classes).
	\end{proof}

	Let $M\preceq \fC$ and $E$ be a bounded invariant equivalence relation on $X$. Then $E^M$ is the relation on $X_M$ (the set of types over $M$ of elements of $X$) defined by $p\Er^M q$ when for some $a\models p$ and $b\models q$ we have $a\Er b$. By Fact \ref{fact: refining by type over M}, $p\Er^M q$ if and only if for every $a\models p$ and $b\models q$ we have $a\Er b$. Hence $E^M$ is an equivalence relation. 

	If $T$ and $M\prec \fC$ are countable, while $E$ is defined on tuples of countable length and $X$ is type-definable, then \cite[Proposition 2.12]{KR16} tells us that the Borel cardinality of the relation $E^M$ on the Polish space $X_M$ does not depend on the choice of $M$, and so the following definition is correct.

	\begin{dfn}\label{dfn: Borel cardinality for strong types}		
		Assume $T$ is countable and $E$ is a bounded invariant equivalence relation on a ($\Aut(\fC)$-invariant) set $X$ of tuples of countable length, where $X$ is type-definable. The {\em Borel cardinality} of $E$ is the Borel cardinality of $E^M$ on $X_M$ for some [every] countable $M \prec \fC$. More generally, if $Y\subseteq X$ is type-definable (over parameters) and $E$-invariant (i.e. a union of some classes of $E$), then the {\em Borel cardinality} of $E \restr_Y$ is defined as the Borel cardinality of $E^M \restr_{Y_M}$.
		\xqed{\lozenge}
	\end{dfn}
	
	\begin{dfn}\label{dfn:Borel strong types}
		Assume $T$ is countable and $E$ is a bounded invariant equivalence relation on tuples of countable length. We say that $E$ is {\em Borel} [resp. {\em analytic}, resp. {\em closed}, etc.] if $E_\emptyset$ is such as a subset of the type space $S(\emptyset)$ in the appropriate number of variables. 
		Similarly, if $Y$ is type-definable (over parameters) and $E$-invariant, then $E \restr_Y$ is {\em Borel} [resp. {\em analytic}, etc.] if $(E \restr_Y)_M$ is such as a subset of $S(M)$ for some [equivalently, any] countable model $M$. \xqed{\lozenge}
	\end{dfn}
	
	Note that if $Y$ is type-definable and $E$-invariant, then $E \restr_Y$ is Borel [resp. analytic, or closed, or $F_\sigma$] if and only if $(E \restr_Y)_M$ is such as a subset of $(Y \times Y)_M$.

	By \cite[Proposition 2.9]{KR16} and the paragraph following \cite[Fact 1.21]{KPR15}, we have
	
	\begin{fct}\label{fct: Borel in various senses}
		 $E$ is type-definable [resp. Borel, or analytic, or $F_\sigma$, or definable] if and only if $E^M$ is closed [resp. Borel, or analytic, or $F_\sigma$, or clopen] in $S(M) \times S(M)$ for some [any] countable model $M$. If $X$ is type-definable, $E^M$ can be equivalently considered only on $X_M \times X_M$ (but then the condition that $E$ is definable should be replaced by the condition that $E$ is relatively definable in $X \times X$). More generally, we have an analogous observation for any type-definable (over parameters) and $E$-invariant subset $Y$ of $X$ and for $E \restr_Y$ in place of $E$.
	\end{fct}
	
	\begin{dfn}\label{dfn:logic topology}
		If $E$ is a bounded invariant equivalence relation on a type-definable set $X$, the \emph{logic topology} on $X/E$ is defined as follows: a subset of $X/E$ is closed when its preimage in $X$ is type-definable (in $X$).
		
		Equivalently, the logic topology is given as the quotient topology induced by $X_M\to X/E$ for any model $M$, i.e.\ it is the quotient topology on $X_M/E^M$, which we naturally identify with $X/E$.
		\xqed{\lozenge}
	\end{dfn}
	
	Directly from the definition of the logic topology, we get the following remark:
	
	\begin{rem}
		If $E\subseteq F$ are bounded invariant equivalence relations on $X$, then the natural map $X/E\to X/F$ is a topological quotient map.
		\xqed{\lozenge}
	\end{rem}

	\begin{fct}
		\label{fct:tdf_t2}
		The logic topology is compact, and it is Hausdorff if and only if $E$ is type-definable.
	\end{fct}
	\begin{proof}
		Compactness of $X/E$ can be found in \cite[Lemma 2.5]{Pil04}, but note that it follows immediately from the fact that $X/E$ is homeomorphic to a quotient of the compact space $X_M$, where $M$ is any model.

		If the logic topology is Hausdorff, then for any model $M$ the relation $E^M$ is closed as a subset of $X_M$ (because it is the preimage of the diagonal in $X/E$ via the natural continuous map $X_M \times X_M \to X/E \times X/E$), which easily implies that $E$ is type-definable. The converse can be found in \cite[Lemma 2.5]{Pil04}, but it follows immediately from the more general fact that the quotient of a compact Hausdorff space by a closed equivalence relation is Hausdorff \cite[3.2.11]{Eng89}. 
	\end{proof}
	
	\begin{dfn}\label{definition: Galois groups}
		The \emph{Lascar strong automorphism group} $\Autf_L(\fC)$ is the subgroup of $\Aut(\fC)$ generated by all $\Aut(\fC/M)$ for $M\preceq \fC$.
		
		The \emph{Lascar strong type $\equiv_L$} is the orbit equivalence relation of $\Autf_L(\fC)$.
		
		The \emph{(Lascar) Galois group} $\Gal(T)$ is the quotient of $\Aut(\fC)$ by $\Autf_L(\fC)$.
		\xqed{\lozenge}
	\end{dfn}
	
	The following fact is folklore, and it easily follows from Fact \ref{fact: refining by type over M}.
	\begin{fct}
		\label{fct:autf_preserves}
		$\Autf_L(\fC)$ preserves all classes of bounded invariant equivalence relations, and $\equiv_L$ is bounded and invariant. In consequence, $\equiv_L$ it is the finest bounded invariant equivalence relation.
\xqed{\lozenge}
	\end{fct}
	
	\begin{dfn}
		The {\em Lascar distance} between tuples $a$ and $b$ is the smallest $n$ such that there are $a=a_0,a_1,\ldots,a_n=b$ such that each pair $a_ia_{i+1}$ extends to an infinite indiscernible sequence (or $\infty$ if such a sequence does not exist).
		\xqed{\lozenge}
	\end{dfn}
	
	\begin{fct}
		Given $a$ and $b$, $a\equiv_L b$ if and only if the Lascar distance between them is finite.
	\end{fct}
	\begin{proof}
		\cite[Fact 1.13]{CLPZ01}.
	\end{proof}
	
	\begin{fct}
		\label{fct:sm_to_gal}
		For any $M\preceq \fC$, if we fix an enumeration $m$ of $M$, we have a commutative diagram
		\begin{center}
			\begin{tikzcd}
				\Aut(\fC)\ar[r]\ar[d] & S_m(M)\ar[d] \\
				\Gal(T) & S_m(M)/{\equiv_L^M}\ar[l]
			\end{tikzcd}
		\end{center}
		where the top arrow is given by $\sigma\mapsto \tp(\sigma(m)/M)$. All arrows are surjective and the bottom arrow is bijective. In particular, we have a surjection $S_m(M)\to \Gal(T)$ given by $\tp(\sigma(m)/M) \mapsto \sigma \Autf_L(\fC)$.
	\end{fct}
	\begin{proof}
		\cite[Fact 2.1]{CLPZ01}
	\end{proof}
	
	\begin{fct}
		The quotient topology on $\Gal(T)$ induced by the surjection $S_m(M)\to \Gal(T)$ from the preceding fact does not depend on $M$ and it makes $\Gal(T)$ a compact (but possibly non-Hausdorff) topological group.
		
		$\Gal(T)$ does not depend on the choice of $\fC$ as a topological group.
	\end{fct}
	\begin{proof}
		\cite[Fact 2.3]{CLPZ01} and \cite{Zie02}.
	\end{proof}

	\begin{rem}
		\label{rem:gal_action}
		$\Gal(T)$ acts on $X/E$ for any bounded invariant equivalence relation $E$ defined on an invariant $X$. If $X=p(\fC)$ for some $p=\tp(a/\emptyset) \in S(\emptyset)$, then the orbit map $r_{[a]_E} \colon \Gal(T) \to X/E$ given by $\sigma \Autf_L(\fC) \mapsto [\sigma(a)]_E$ is a topological quotient map.
	\end{rem}
	\begin{proof}
		$\Aut(\fC)$ acts on $X$, and by invariance of $E$, it also acts on $X/E$. By Fact~\ref{fct:autf_preserves}, this action factors through $\Gal(T)$, which gives us the first part. The second one follows easily from the definition of the logic topology and the topology on $\Gal(T)$.
	\end{proof}

	\begin{dfn}
		The group $\Autf_{KP}(\fC)$ of Kim-Pillay strong automorphisms consists of those automorphisms which fix setwise all classes of all bounded, $\emptyset$-type-definable equivalence relations.
		
		The \emph{Kim-Pillay strong type $\equiv_{KP}$} is the orbit equivalence relation of $\Autf_{KP}(\fC)$.
		
		The group $\Gal_{KP}(T)$ is the quotient $\Aut(\fC)/\Autf_{KP}(\fC)$.
		\xqed{\lozenge}
	\end{dfn}
	
	\begin{fct}
		$\equiv_{KP}$ is the finest bounded, $\emptyset$-type-definable equivalence relation.
	\end{fct}
	\begin{proof}
		\cite[Fact 1.4]{CLPZ01}
	\end{proof}

	By $\Gal_0(T)$ we will denote the closure of the identity in $\Gal(T)$.
	
	\begin{fct}
		\label{fct:galkp_quot}
		$\Gal_{KP}(T)$ is the quotient of $\Gal(T)$ by $\Gal_0(T)$. As a consequence, it is a compact Hausdorff topological group, and it does not depend on $\fC$ (as a topological group). (We call the induced topology the logic topology.)
	\end{fct}
	\begin{proof}
		\cite[Fact 2.3]{CLPZ01}
	\end{proof}
	
	\begin{rem}\label{rem: GalKP is Polish}
		If $T$ is countable, then $\Gal_{KP}(T)$ is a compact, Polish group.
	\end{rem}
	\begin{proof}
		Since $T$ is countable, we can choose a countable model $M$ with an enumeration $m$, and then $S_m(M)$ is compact Polish, and we finish using Facts \ref{fct: preservation of metrizability}, \ref{fct:galkp_quot}, and the fact that $\Gal(T)$ is a continuous image of $S_m(M)$. 
	\end{proof}
	
	Analogously to Remark \ref{rem:gal_action}, we have
	
	\begin{rem}
		\label{rem:galkp_action}
		$\Gal_{KP}(T)$ acts on $X/E$ for any invariant equivalence relation $E$ coarser than $\equiv_{KP}$ defined on an invariant set $X$. If $X=p(\fC)$ for some $p=\tp(a/\emptyset) \in S(\emptyset)$, then the orbit map $r_{[a]_E} \colon \Gal_{KP}(T) \to X/E$ given by $\sigma \Autf_{KP}(\fC) \mapsto [\sigma(a)]_E$ is a topological quotient map.
		\xqed{\lozenge}
	\end{rem}
	
	In Section 5 of \cite{KPS13} it is explained how to define the Borel cardinalities of $\Gal(T)$ and $\Gal_0(T)$. Briefly, let $m$ be an enumeration of a countable model $M$ of a countable theory $T$.  
	Fact \ref{fct:sm_to_gal} identifies $\Gal(T)$ with $S_m(M)/{\equiv_L^M}$, and we define the {\em Borel cardinality} of $\Gal(T)$ as the Borel cardinality of $S_m(M)/{\equiv_L^M}$. Similarly, the {\em Borel cardinality} of $\Gal_0(T)$ is defined as the Borel cardinality of $S_{m}^{KP}(M)/{\equiv_L^M}$, where $S_{m}^{KP}(M):=\{ \tp(n/M): n \equiv_{KP} m\}$.

	\section{Relations coarser than the Kim-Pillay strong type}\label{section: relations coarser than the Kim-Pillay strong type}
	In this section, we will discuss the relations coarser than the Kim-Pillay strong type. The main point is that --- unlike the general case --- we do not need to construct any group using topological dynamics: we can just use $\Gal_{KP}(T)$ instead. This makes the problem much simpler, and allows us to focus only on the descriptive set theoretical aspect of the problem, which will roughly translate into the general case. Note that this approach applies to all strong types if the underlying theory is G-compact (which includes all stable and, more generally, simple theories).

	\begin{lem}
		\label{lem:easy}
		Suppose we have a commutative diagram
		\begin{center}
			\begin{tikzcd}
				A\ar[d, two heads]\ar[r, two heads]&G\ar[d] \\
				C\ar[r]&Q
			\end{tikzcd}
		\end{center}
		
		where:
		\begin{itemize}
			\item
			$A$, $C$ and $G$ are compact Polish spaces,
			\item
			the surjections $A \to C$ and $A \to G$ are continuous.
		\end{itemize}
		Denote by $E_C$ and $E_G$ the equivalence relations on $C$ and $G$ (respectively) induced by equality on $Q$. Then:
		\begin{enumerate}
			\item
			$E_G$ is closed [resp. Borel, or analytic, or $F_\sigma$, or clopen (equivalently, with open classes)] if and only if $E_C$ is such,
			\item
			$E_G\sim _B E_C$.
		\end{enumerate}
	\end{lem}
	\begin{proof}
		Denote by $E_A$ the equivalence relation on $A$ induced by equality on $Q$ via the composed map $A\to Q$.
				
		(1) Since $E_A$ is the preimage of each of $E_C$ and $E_G$ by a continuous surjection between compact Polish spaces, by Remark \ref{rem: continuous surjection is closed} and the comments following Fact \ref{fct:borel_section}, we conclude that closedness [resp. Borelness, or analyticity, or being $F_\sigma$, or being clopen] of $E_A,E_C$ and $E_G$ are all equivalent.
				
		(2) 
		It is clear that the top and the left arrow are continuous, surjective reductions of $E_A$ to $E_G$ and $E_A$ to $E_C$, respectively. So $E_G \sim_B E_A \sim_B E_C$ by Fact \ref{fct:borel_section}.
	\end{proof}
	
	The following theorem is a prototype for the main result (Theorem~\ref{thm:main}).
	\begin{thm}
		\label{thm:main_over_KP}
		Suppose $E$ is a strong type defined on $p(\fC)$ for some $p\in S(\emptyset)$ (in countably many variables, in an arbitrary countable theory) and $E$ is refined by $\equiv_{KP}$. Fix any $a\models p$. 
		
		Consider the orbit map $r_{[a]_E}\colon \Gal_{KP}(T)\to p(\fC)/E$ given by $\sigma\Autf_{KP}(\fC)\mapsto [\sigma(a)]_E$ (the orbit map of the natural action of $\Gal_{KP}(T)$ on $p(\fC)/E$ introduced in Remark~\ref{rem:galkp_action}), and put $H=\ker r_{[a]_E}:=r_{[a]_E}^{-1}[a/E]$. Then:
		\begin{enumerate}
			\item
			$H\leq \Gal_{KP}(T)$ and the fibers of $r_{[a]_E}$ are the left cosets of $H$,
			\item
			 $r_{[a]_E}$ is a topological quotient mapping, and so $p(\fC)/E$ is homeomorphic to $\Gal_{KP}(T)/H$,
			\item
			$E$ is type-definable [resp. Borel, or analytic, or $F_\sigma$, or relatively definable on $p(\fC) \times p(\fC)$] if and only if $H$ is closed [resp. Borel, or analytic, or $F_\sigma$, or clopen],
			\item
			$E_H\sim_B E$, where $E_H$ is the relation of lying in the same left coset of $H$.
		\end{enumerate}
	\end{thm}
	\begin{proof}
		The first point is immediate by the fact that $r_{[a]_E}$ is an orbit map, namely the fibers of $r_{[a]_E}$ are the left costs of the stabilizer of $[a]_E$ (under the action of $\Gal_{KP}(T)$) which is exactly $H$. The second point follows from the first one and Remark~\ref{rem:galkp_action}.
		
		Let $M$ be a countable model containing $a$, and let $m\supseteq a$ be an enumeration of $M$. Then we have a commutative diagram.
		\begin{center}
			\begin{tikzcd}
				S_m(M) \ar[r,two heads] \ar[d,two heads] & \Gal_{KP}(T) \ar[d,"r_{[a]_E}",two heads] \\
				S_a(M) \ar[r,two heads] & {[a]}_\equiv/E
			\end{tikzcd}
		\end{center}
		The top arrow is defined in the same way as the map to $\Gal(T)$ given by Fact~\ref{fct:sm_to_gal}. The left arrow is the restriction map, and the bottom one is the quotient map given by Fact~\ref{fct:stypes_from_types}.

		It is easy to check that this diagram is commutative and consists of continuous maps. Moreover, $S_m(M), S_a(M)$ and $\Gal_{KP}(T)$ are all compact Polish (see Remark \ref{rem: GalKP is Polish}).

		Let $f \colon \Gal_{KP}(T) \times \Gal_{KP}(T) \to \Gal_{KP}(T)$ be given by $f(xy)=y^{-1}x$. Then $E_H = f^{-1}[H]$. Hence, since $f$ is a continuous surjection between compact Polish spaces, using Remark \ref{rem: continuous surjection is closed} and 
		the comments following Fact \ref{fct:borel_section}, we get that $E_H$ is closed [resp. Borel, or analytic, or $F_\sigma$, or clopen] if and only if $H$ is.
		
		On the other hand, we can apply Lemma~\ref{lem:easy}, and the rest of the conclusion follows in a straightforward manner using Fact \ref{fct: Borel in various senses}.
	\end{proof}
	
	We immediately obtain Fact~\ref{fct:main_KPR} for strong types coarser than $\equiv_{KP}$.
	\begin{cor}
		\label{cor:main_over_KP}
		Assume $T$ is countable. Let $E$ be a Borel (or, more generally, analytic) strong type on $p(\fC)$ for some $p\in S(\emptyset)$ (in countably many variables). Assume that $E$ is coarser than $\equiv_{KP}$. Then exactly one of the following conditions holds:
		\begin{enumerate}
			\item
			$p(\fC)/E$ is finite and $E$ is relatively definable,
			\item
			$|p(\fC)/E|=2^{\aleph_0}$ and $E$ is type-definable and smooth,
			\item
			$|p(\fC)/E|=2^{\aleph_0}$ and $E$ is neither type-definable nor smooth.
		\end{enumerate}
	\end{cor}
	\begin{proof}
		Apply Theorem~\ref{thm:main_over_KP}, and then use Fact~\ref{fct:trich_polish} for $G:=\Gal_{KP}(T)$ and $H$. For more details, see the proof of Corollary \ref{cor:main_KPR}.
	\end{proof}
	
	For arbitrary strong types, we do not have the action of $\Gal_{KP}(T)$ on $[a]_{\equiv}/E$, and so we cannot apply Lemma~\ref{lem:easy} directly. Instead, we have an action of the group $\Gal(T)$ which in general is not Hausdorff (so not Polish). The proof of \hyperref[mainthm]{Main Theorem} will consist of finding a compact Polish extension $\hat G$ of $\Gal(T)$ (as a topological group and as a ``Borel quotient group"). In place of Lemma~\ref{lem:easy}, we will use their variants, Lemmas~\ref{lem:main_1}, \ref{lem:main_2}, which we will apply to two distinct diagrams.
	
	The analogue of Corollary~\ref{cor:main_over_KP} will then naturally follow in the form of Corollary~\ref{cor:main_KPR}.
	
	To construct $\hat G$, we will revisit and refine the topological dynamical methods developed in \cite{KPR15}.

	\section{Topological dynamics}\label{section: topological dynamics}

	In the first subsection, we list the necessary definitions and facts from general topological dynamics. The following two subsections are devoted to Rosenthal compacta and tame dynamical systems. All of this is standard knowledge presented in the form and generality suitable for our applications.

	\subsection{Flows, Ellis semigroups, and Ellis groups}
	\label{ssec:prel_topdyn}
	\begin{dfn}
		By a \emph{dynamical system}, in this paper, we mean a pair $(G,X)$, where $G$ is an abstract group acting by homeomorphisms on a compact Hausdorff space $X$.
		\xqed{\lozenge}
	\end{dfn}

	\begin{dfn}
		If $(G,X)$ is a dynamical system, then its \emph{Ellis (or enveloping) semigroup} $EL=E(G,X)$ is the (pointwise) closure in $X^X$ of the set of functions $\pi_g\colon x\mapsto g\cdot x$ for $g\in G$. (We frequently slightly abuse the notation and write $g$ for $\pi_g$, treating $G$ as if it was a subset of $E(G,X)$.)
		\xqed{\lozenge}
	\end{dfn}
	
	\begin{fct}
		If $(G,X)$ is a dynamical system, then $EL$ is a compact left topological semigroup (i.e.\ it is a semigroup with the composition as its semigroup operation, and the composition is continuous on the left). It is also a $G$-flow with $g\cdot f:= \pi_g \circ f$.
	\end{fct}
	\begin{proof}
		Straightforward ($X^X$ itself is already a compact left topological semigroup, and it is easy to check that $EL$ is a closed subsemigroup).
	\end{proof}

	\begin{fct}[minimal ideals and the Ellis group]
		Suppose $S$ is a compact Hausdorff left topological semigroup (e.g.\ the enveloping semigroup of a dynamical system). Then $S$ has a minimal left ideal $\cM$. Furthermore, for any such ideal $\cM$:
		\begin{enumerate}
			\item
			$\cM$ is closed,
			\item
			for any element $a\in \cM$, $\cM=Sa=\cM a$,
			\item
			$\cM=\bigsqcup_u u\cM$, where $u$ runs over all idempotents in $\cM$ (i.e.\ elements such that $u\cdot u=u$) --- in particular, $\cM$ contains idempotents,
			\item
			for any idempotent $u\in \cM$, the set $u\cM$ is a subgroup of $S$ with the identity element $u$ (note that $u$ is usually \emph{not} the identity element of $S$ --- indeed, $S$ need not have an identity at all).
		\end{enumerate}
		Moreover, all the groups $u\cM$ (where $\cM$ ranges over all minimal left ideals and $u$ over idempotents in $\cM$) are isomorphic. The isomorphism type of all these groups is called the {\em Ellis} (or {\em ideal}) group of $S$; if $S=E(G,X)$, we call this group the {\em Ellis group} of the flow $(G,X)$.
	\end{fct}
	\begin{proof}
		Classical. See e.g.\ Corollary 2.10 and Propositions 3.5 and 3.6 of \cite{Ell69}, or Proposition 2.3 of \cite[Section I.2]{Gl76}.
	\end{proof}

	\begin{fct}
		\label{fct:derived_quotient}
		Suppose $G$ is a compact (possibly non-Hausdorff) semitopological group (i.e. with separately continuous multiplication). Denote by $H(G)$ the intersection $\bigcap_V \overline V$, where $V$ varies over the open neighborhoods of the identity in $G$.
		
		Then $H(G)$ is a closed normal subgroup of $G$ and $G/H(G)$ is a compact Hausdorff group (in fact, $G/H(G)$ is the universal Hausdorff quotient of $G$).
	\end{fct}
	\begin{proof}
		This is essentially the content of Lemma 1.8 and Theorem 1.9 in Section IX.1. of \cite{Gl76}.
	\end{proof}
	
	\begin{fct}[$\tau$-topology on the Ellis group in an enveloping semigroup]
		\label{fct:tau_top}
		Consider the Ellis semigroup $EL$ of a dynamical system $(G,X)$. Fix any minimal left ideal $\cM$ of $EL$ and an idempotent $u\in \cM$.
		\begin{enumerate}
			\item
			For each $a\in EL$, $B\subseteq EL$, we write $a\circ B$ for the set of all limits of nets $(g_ib_i)_i$, where $g_i\in G$ are such that $\pi_{g_i}=g_i\cdot \id \to a$, and $b_i\in B$.
			\item
			For any $p,q\in EL$ and $A\subseteq EL$, we have:
			\begin{itemize}
				\item
				$p\circ(q\circ A)\subseteq (pq)\circ A$,
				\item
				$pA\subseteq p\circ A$,
				\item
				$p\circ A=p\circ \overline A$,
				\item
				$p\circ A$ is closed,
				\item
				if $A\subseteq \cM$, then $p\circ A\subseteq \cM$.
			\end{itemize}
			\item
			The formula $\cl_\tau(A):=(u\cM)\cap (u\circ A)$ defines a closure operator on $u\cM$. It can also be (equivalently) defined as $\cl_\tau(A)=u(u\circ A)$. We call the topology on $u\cM$ induced by this operator the {\em $\tau$ topology}.
			\item
			\label{it:lem:tau_top:nearlycont}
			If $(f_i)_i$ (a net in $u\cM$) converges to $f\in \overline{u\cM}$ (the closure of $u\cM$ in $EL$), then $(f_i)_i$ converges to $uf$ in the $\tau$-topology.
			\item
			The $\tau$-topology on $u\cM$ refines the subspace topology inherited from $EL$.
			\item
			$u\cM$ with the $\tau$ topology is a compact $T_1$ semitopological group. Consequently, $u\cM/H(u\cM)$ is a compact, Hausdorff group (see Fact~\ref{fct:derived_quotient}).
		\end{enumerate}
	\end{fct}
	\begin{proof}
		Much of these facts is contained in \cite[Section IX.1]{Gl76}. There, the author considers the special case of $EL=\beta G$ and defines $\circ$ in a slightly different way (but both definitions are equivalent in this special case). However, as pointed out in \cite[Section 2]{KP17} and \cite[Section 1.1]{KPR15}, many of the proofs from \cite[Section IX.1]{Gl76} go through in the general context. 
		Otherwise, we use straightforward calculations with nets. 
		See the discussion following Definition 2.1 of \cite{KP17} (e.g. for a proof of the first item in (2)).
	\end{proof}
	
	\begin{fct}
		\label{fct:strange_cont}
		The function $\xi\colon \overline {u\cM}\to u\cM$ (where $\overline{u\cM}$ is the closure of $u\cM$ in the topology of $EL$) defined by the formula $f\mapsto uf$ has the property that for any continuous function $\zeta\colon u\cM\to X$, where $X$ is a regular topological space and $u\cM$ is equipped with the $\tau$-topology, the composition $\zeta\circ \xi \colon \overline {u\cM}\to X$ is continuous, where $\overline{u\cM}$ is equipped with subspace topology from $EL$. In particular, the map $\overline{u\cM}\to u\cM/H(u\cM)$ given by $f \mapsto uf/H(u\cM)$ is continuous.
	\end{fct}
	\begin{proof}
		This is \cite[Lemma 3.1]{KPR15}.
	\end{proof}

\subsection{Rosenthal compacta, independent sets, and \texorpdfstring{$\ell^1$}{l1} sequences}\label{section: Rosenthal compacta}
	Here, we will discuss selected properties of Rosenthal compacta. For a broader exposition, refer to \cite{Debs14}.
	
	\begin{dfn}
		Given a topological space $X$, we say that a function $X\to \bbR$ is of \emph{Baire class 1} if it is the pointwise limit of a sequence of continuous real-valued functions.	
		We denote by $\cB_1(X)$ the set of all such functions.\xqed{\lozenge}
	\end{dfn}

	\begin{dfn}
		A compact, Hausdorff space $K$ is a \emph{Rosenthal compactum} if it embeds homeomorphically into $\cB_1(X)$ for some Polish space $X$, where $\cB_1(X)$ is equipped with the pointwise convergence topology.
		\xqed{\lozenge}
	\end{dfn}
	\begin{dfn}
		A {\em Fréchet} (or {\em Fréchet-Urysohn}) space is a topological space in which any point in the closure of a given set $A$ is the limit of a sequence of elements of $A$.
		\xqed{\lozenge}
	\end{dfn}
	
	\begin{fct}\label{fct: Rosnthal implies Frechet}
		Rosenthal compacta are Fréchet.
	\end{fct}
	\begin{proof}
		\cite[Theorem 4.1]{Debs14}.
	\end{proof}

	\begin{fct}
		\label{fct:bft}
		Suppose $X$ is a compact metric space and $A\subseteq C(X)$ is a family of $0-1$ valued functions (i.e.\ characteristic functions of clopen subsets of $X$). Put $\mathcal A:=\{U\subseteq X\mid \chi_U\in A \}$. The following are equivalent:
		\begin{itemize}
			\item
			$\overline A\subseteq {\bbR}^X$ is Fréchet (equivalently, Rosenthal),
			\item
			$\mathcal A$ contains no infinite independent family, i.e. $\mathcal A$ contains no family $(A_i)_{i \in \bbN}$ such that for every $I \subseteq \bbN$ the intersection $\bigcap_{i \in I} A_i \cap \bigcap_{i \in \bbN \setminus I} A_i^c$ is nonempty.
		\end{itemize}
	\end{fct}
	\begin{proof}
		$A$ is clearly pointwise bounded, so by \cite[Corollary 4G]{BFT78}, $A$ is relatively compact in $\cB_1(X)$ (which is equivalent to the first condition) if and only if it satisfies the condition (vi) from \cite[Theorem 2F]{BFT78}, which for $0-1$ functions on a compact space reduces to the statement that for each sequence $(a_n)$ of elements of $A$ there is some $I\subseteq \bbN$ for which there is no $x\in X$ such that $a_n(x)=1$ if and only if $n\in I$.
		This is clearly equivalent to the second condition.
	\end{proof}
	
	The next definition is classical and can be found for example in \cite[Section 5]{Koh95}.
	\begin{dfn}
		If $(f_n)_{n\in {\bbN}}$ is a sequence of elements in a Banach space, we say that it is an {\em $\ell^1$ sequence} if it is bounded and there is a constant $\theta>0$ such that for any scalars $c_0,\ldots,c_n$ we have the inequality
		\[
		\theta\cdot \sum_{i=0}^n \lvert c_i\rvert < \left \lVert \sum_{i=0}^n c_i f_i\right\rVert.
		\]
		(This is equivalent to saying that $e_n\mapsto f_n$ extends to a topological isomorphism of $\ell^1$ and the closed span of $(f_n)_n$ (in the norm topology), where $e_n$ are the standard basis vectors.)
		\xqed{\lozenge}
	\end{dfn}
	
	In fact, $\ell^1$ sequences are very intimately related to ``independent sequences" (via the Rosenthal's dichotomy).
	The following is a simple example of this relationship:
	
	\begin{fct}
		\label{fct:ind_untame}
		Suppose $X$ is a compact, Hausdorff space and $(A_n)_n$ is an independent sequence of clopen subsets of $X$. Then $(\chi_{A_n})_n$ is an $\ell^1$ sequence in the Banach space $C(X)$ (with the supremum norm).
	\end{fct}
	\begin{proof}
		Fix any sequence $c_0,\ldots,c_n$ of real numbers. Write $[n]$ for $\{0,\ldots,n\}$ and put $f:=\sum_{i\in [n]} c_i\chi_{A_i}$. Let $I:=\{i\in [n]\mid c_i\geq 0 \}$. Assume without loss of generality that $\sum_{i\in I} c_i\geq -\sum_{i\in [n]\setminus I} c_i$ (the other case is analogous). Then for any $x\in \bigcap_{i\in I} A_i\cap \bigcap_{i\in [n]\setminus I} A_i^c$ we have $f(x)=\sum_{i\in I} c_i\geq \frac{1}{2} \sum_{i\in [n]} \lvert c_i\rvert$.
	\end{proof}

	\subsection{Tame dynamical systems}\label{section: tame systems}

	\begin{dfn}
		If $(G,X)$ is a dynamical system and $f\in C(X)$, then we say that $f$ is a \emph{tame function} if for every sequence $(g_n)_n$ of elements of $G$, $(f\circ g_n)_n$ is not an $\ell^1$ sequence.
		
		We say that $(G,X)$ is a \emph{tame dynamical system} if every $f\in C(X)$ is tame.
		\xqed{\lozenge}
	\end{dfn}
	
	\begin{rem}
		\label{rem:dfn_equiv}
		The notion of tame dynamical system is due to Kohler \cite{Koh95}. She used the adjective ``regular" instead of (now established) ``tame", and formulated it for actions of $\bbN$ on metric compacta, but we can apply the same definition to arbitrary group actions on compact spaces.
		
		Some authors use different (but equivalent) definitions of tame function or tame dynamical system. For example, \cite[Fact 4.3 and Proposition 5.6]{GM12} yields several equivalent conditions for tameness of a function (including the definition given above and \cite[Definition 5.5]{GM12}). By this and \cite[Corollary 5.8]{GM12}, we obtain equivalence between our definition of tame dynamical system and \cite[Definition 5.2]{GM12}.\xqed{\lozenge}
	\end{rem}

	The condition in the following fact can be used as a definition of tameness for metric dynamical systems.
	\begin{fct}\label{fct: metric tameness}
		If $(G,X)$ is a metric dynamical system and $f\in C(X)$, then $f$ is tame if and only if the pointwise closure $\overline{\{f\circ g\mid g\in G \} }\subseteq {\bbR}^X$ consists of Baire class 1 functions (equivalently, it is a Rosenthal compactum).
	\end{fct}
	\begin{proof}
		It follows immediately from \cite[Fact 4.3 and Proposition 4.6]{GM12}.
	\end{proof}

	\begin{fct}
		\label{fct:tame_closed}
		For any dynamical system, the tame functions form a closed subalgebra of $C(X)$ (with pointwise multiplication and norm topology).
	\end{fct}
	\begin{proof}
		First, by Remark~\ref{rem:dfn_equiv}, tame functions in $(G,X)$ satisfy \cite[Definition 5.5]{GM12}, i.e.\ for every $f$ tame in $X$ there is a tame dynamical system $(G,Y_f)$ and an epimorphism $\phi_f\colon X\to Y_f$ such that $f=\phi_f^*(f'):=f'\circ \phi_f$ for some $f'\in C(Y_f)$.
		
		Since tame dynamical systems are closed under subsystems and under arbitrary products (\cite[Lemma 5.4]{GM12}), there is a universal $Y$ for all tame functions $f$, i.e.\ such that the set of all tame functions in $(G,X)$ is exactly the image of $\phi^*\colon C(Y)\to C(X)$, where $\phi\colon X\to Y$ is an epimorphism and $Y$ is tame (just take $\phi\colon X\to \prod_f Y_f$ to be the diagonal of $\phi_f$, and take $Y:=\phi[X] \subseteq \prod_f Y_f$). 
		
		Since $C(Y)$ is a Banach algebra and $\phi^*$ is a homomorphism and an isometric embedding (as $\phi$ is onto), the fact follows.
	\end{proof}
	
	\begin{cor}
		\label{cor:tame_dense}
		If $(G,X)$ is a dynamical system and $\mathcal A\subseteq C(X)$ is a family of functions separating points, then $(G,X)$ is tame if and only if every $f\in \mathcal A$ is tame.
	\end{cor}
	\begin{proof}
		The implication $(\leftarrow)$ is obvious.
		
		$(\rightarrow)$.
		Since constant functions are trivially tame, by the assumption and the Stone-Weierstrass theorem, it follows that tame functions are dense in $C(X)$, and thus the conclusion follows immediately from Fact~\ref{fct:tame_closed}.
	\end{proof}

	\begin{fct}
		\label{fct:tame_preserved}
		Suppose $(G,X)$ is a tame dynamical system. Then the following dynamical systems are tame:
		\begin{itemize}
			\item
			$(H,X)$, where $H\leq G$,
			\item
			$(G,X_0)$, where $X_0\subseteq X$ is a closed invariant subspace,
			\item
			$(G,Y)$, where $Y$ is a $G$-equivariant quotient of $X$.
		\end{itemize}
	\end{fct}
	\begin{proof}
		The first bullet is trivial. The second follows from the Tietze extension theorem. For the third, note that any potentially untame function on $Y$ can be pulled back to $X$.
	\end{proof}

	The following is a dynamical variant of Bourgain-Fremlin-Talagrand dichotomy, slightly modified for our needs from \cite[Theorem 3.2]{GM06}.
	\begin{prop}[Dynamical BFT dichotomy]
		\label{prop:dyn_BFT}
		Suppose $X$ is a totally disconnected metric compactum, and $G$ acts on $X$ by homeomorphisms. Then the following are equivalent:
		\begin{enumerate}
			\item
			\label{it:prop:dyn_BFT:untame}
			$(G,X)$ is untame,
			\item
			\label{it:prop:dyn_BFT:clopen}
			there is a clopen set $U$ and a sequence $(g_n)_{n\in \bbN}$ of elements of $G$ such that the sets $g_n U$ are independent,
			\item
			\label{it:prop:dyn_BFT:betaN}
			$EL:=E(G,X)$ contains a homeomorphic copy of $\beta \bbN$,
			\item
			\label{it:prop:dyn_BFT:large}
			$\lvert EL\rvert=2^{2^{\aleph_0}}$,
			\item
			\label{it:prop:dyn_BFT:Fréchet}
			$EL$ is not Fréchet,
			\item
			\label{it:prop:dyn_BFT:Rosenthal}
			$EL$ is not a Rosenthal compactum.
%			\item
%			\label{it:prop:dyn_BFT:Baire}
%			$EL$ is not contained in $\cB_1(X,X)$.
		\end{enumerate}
		If $X$ is not necessarily totally disconnected, all conditions but \eqref{it:prop:dyn_BFT:clopen} are equivalent.
	\end{prop}
	\begin{proof}
		The equivalence of all conditions but (2)
		is proved in \cite[Theorem 3.2]{GM06} (based on the Bourgain-Fremlin-Talagrand dichotomy). 
		For the reader's convenience, we will prove here that all conditions (including \eqref{it:prop:dyn_BFT:clopen}) are equivalent in the totally disconnected case (the case which appears in our model-theoretic applications).
		
		$(1) \rightarrow (2)$. Since the characteristic functions of clopen subsets of $X$ are continuous and separate points in $X$, by \eqref{it:prop:dyn_BFT:untame} and Corollary~\ref{cor:tame_dense}, the characteristic function $\chi_U$ is not tame for some clopen $U\subseteq X$. By Fact \ref{fct: metric tameness}, this is equivalent to the fact that $\overline{\{\chi_{gU} \mid g \in G\}}$ is not a Rosenthal compactum. Hence, Fact \ref{fct:bft} implies that some family $\{g_n U: n \in \bbN\}$ (with $g_n \in G$) is independent.
		
		$(2) \rightarrow (1)$. The reversed argument works. Alternatively, it follows immediately from Fact \ref{fct:ind_untame}.
		
		$(2) \rightarrow (3)$. Let $(g_n)$ be a sequence of elements of $G$ such that the sets $g_n U$ are independent. By the universal property of $\beta \bbN$, we have the continuous function $\beta \colon \beta \bbN\to EL$ given by $\mathcal F\mapsto \lim_{n\to \mathcal F} g_n^{-1}$. It remains to check that $\beta$ is injective. Consider two distinct ultrafilters $\mathcal F_1$ an $\mathcal F_2$ on $\bbN$. Choose $F \in \mathcal F_1 \setminus \mathcal F_2$. By the independence of the $g_n U$, we can find $x \in \bigcap_{n \in F} g_nU \cap \bigcap_{n \in \bbN \setminus F} g_nU^c$. It suffices to show that $\beta(\mathcal F_1)(x) \ne \beta(\mathcal F_2)(x)$. Note that $\{n \in \bbN \mid g_n^{-1}x \in U^c\} = \bbN \setminus F \notin \mathcal F_1$ and $U^c$ is open, so $\beta(\mathcal F_1)(x) \in U$. Similarly, $\beta(\mathcal F_2)(x) \in U^c$, and we are done.
		
		$(3) \rightarrow (4)$. The group $\{\pi_g \mid g \in G\}$ is contained in the Polish group $\mbox{Homeo}(X,X)$ of all homeomorphisms of $X$ equipped with the uniform convergence topology. So $\{\pi_g \mid g \in G\}$ is separable in the inherited topology, and so also in the pointwise convergence topology (which is weaker). Therefore, $EL =\overline{ \{\pi_g \mid g \in G\}}$ is of cardinality at most $2^{2^{\aleph_0}}$. On the other hand, $|\beta\bbN| = 2^{2^{\aleph_0}}$. Hence, $|EL|= 2^{2^{\aleph_0}}$.
		
		$(4) \rightarrow (5)$. If $EL$ is Fréchet, then, using the above observation that $\{\pi_g \mid g \in G\}$ is separable, we get that $|EL| = 2^{\aleph_0}$.
		
		$(5) \rightarrow (6)$. This is Fact \ref{fct: Rosnthal implies Frechet}.
		
		$(6) \rightarrow (1)$. Embed homeomorphically $X$ in $\bbR^\bbN$. Then $EL$ embeds homeomorphically in $\bbR^{X \times \bbN}$ via the map $\Phi$ given by $\Phi(f)(x,i):=f(x)(i)$. Take $f \in EL$, and let $\pi_i \colon X \to \bbR$ be the projection to the $i$-th coordinate, i.e. $\pi_i(x):=x(i)$. Suppose $(G,X)$ is tame. Then $\pi_i \circ f \in \cB_1(X)$ by Fact \ref{fct: metric tameness}, so for every $i \in \bbN$ there is a sequence of continuous functions $f^i_n \colon X \to \bbR$ such that $\lim_n f^i_n = \pi_i \circ f$. Define $f_n \in \bbR^{X \times \bbN}$ by $f_n(x,i) := f^i_n(x)$. Then all $f_n$'s are continuous and $\Phi(f)= \lim_n f_n$.
		So $\Phi[EL]$ is a compact subset of $\cB_1(X \times \bbN)$, i.e. $EL$ is Rosenthal.
	\end{proof}

	\begin{fct}
		\label{fct:tame_borel}
		If $(G,X)$ is a metric dynamical system, then $(G,X)$ is tame if and only if all functions in $E(G,X)$ are Borel measurable.
	\end{fct}
	\begin{proof}
		By Proposition~\ref{prop:dyn_BFT}, if $(G,X)$ is tame, $E(G,X)$ is Fréchet. Since the pointwise limit of a sequence of continuous functions between Polish spaces is always Borel, it follows that $E(G,X)$ consists of Borel functions.
		
		In the other direction, since $X$ is Polish, there are at most $2^{\aleph_0}$ many Borel functions $X\to X$. In particular, if $E(G,X)$ consists of Borel functions, $\lvert E(G,X)\rvert\leq 2^{\aleph_0}<2^{2^{\aleph_0}}$, which implies tameness by Proposition~\ref{prop:dyn_BFT}.
	\end{proof}

	\section{Independence, tameness and ambition}\label{section: independence, tameness and ambition}
	In this section, we discuss the relationship between model-theoretic NIP and dynamical tameness. A relationship between the Bourgain-Fremlin-Talagrand dichotomy and NIP seems to have been first noticed independently in \cite{CS18}, \cite{Ib16}, and \cite{Kha14}; see also \cite{Sim15} and \cite{KhP17} for related research. Many statements in this section appear to be folklore, but we have not found them stated and proved in this form, so we present them along with their proofs, as they are interesting in their own right. The introduced notions of tame models and ambitious models seem to be new. Ambitious models will be essential later.

	\begin{dfn}
		If $A,B\subseteq \fC$, then we say that a formula $\varphi(x,y)$ has the \emph{ independence property} (IP) on $A\times B$ if there is an infinite sequence $(b_n)$ of elements of $B$ such that $\varphi(\fC,b_n)\cap A$ are independent subsets of $A$. Otherwise, we say that it \emph{has NIP} on $A\times B$.
		
		We say that $\varphi$ \emph{has IP} if it has IP on the whole $\fC$, otherwise we say that it has NIP.
		
		We say that \emph{$T$ has NIP} if every formula has NIP. Otherwise, we say that \emph{$T$ has IP}.
		\xqed{\lozenge}
	\end{dfn}
	
	\begin{rem}
		Note that if $A$ and $B$ are type-definable, then in the above definition we can assume without loss of generality that the sequence $(b_n)$ is indiscernible (by Ramsey's theorem and compactness).\xqed{\lozenge}
	\end{rem}
	
	\begin{dfn}
		We say that a formula $\varphi(x,y)$ is \emph{tame} if for every small model $M$ and $b\in M$, the characteristic function of $[\varphi(x,b)]\subseteq S_x(M)$ is tame in $(\Aut(M),S_x(M))$.
		
		Similarly, if $A$, $B$ are $\emptyset$-type-definable sets, we say that $\varphi(x,y)$ is {\em tame} on $A\times B$ if for every small model $M$ and $b\in B(M)$, the characteristic function of $[\varphi(x,b)]\cap A_M\subseteq A_M$ is tame in $(\Aut(M),A_M)$ (where $A_M\subseteq S(M)$ is the space of types of the elements of $A$).
		\xqed{\lozenge}
	\end{dfn}
	
	Note that tameness of $\varphi(x,y)$ does not change when we add dummy variables, even allowing infinite sequences of variables.
	
	\begin{lem}\label{lem:NIP_tame}
		[For any $\emptyset$-type-definable sets $A,B$] $\varphi(x,y)$ is NIP [on $A\times B$] if and only if $\varphi(x,y)$ is tame [on $A\times B$].
	\end{lem}
	\begin{proof}
		For simplicity, we will treat the absolute case here. The relative (i.e.\ $A\times B$) case is proved similarly.
		
		If $\varphi(x,y)$ has IP, there is an indiscernible sequence $(b_n)$ witnessing that, and we can find a small model $M$ which contains $(b_n)$, and such that 
		all $b_n$'s lie in a single orbit under $\Aut(M)$.
		It follows from Fact~\ref{fct:ind_untame} that $\varphi$ is untame (which is witnessed in $(\Aut(M),S_x(M))$).
		
		In the other direction, suppose $\varphi(x,y)$ is untame. Fix a small model $M$ and $b\in M$ witnessing that. Then we have a sequence $(\sigma_n)_n$ in $\Aut(M)$ such that $\sigma_n\cdot \chi_{[\varphi(x,b)]}$ is an $\ell^1$ sequence.
		
		Let $\Sigma\leq \Aut(M)$ be the group generated by all $\sigma_n$'s and $B_0:=\Sigma\cdot b$. Then $B_0$ is countable and $S_{\varphi}(B_0)$ is a totally disconnected, compact metric space. Moreover, the characteristic function of $[\varphi(x,b)]\subseteq S_\varphi(B)$ is untame with respect to $(\Sigma,S_{\varphi}(B))$. Then, by Prop~\ref{prop:dyn_BFT}, there is a $\varphi$-formula $\psi$ with IP. Since NIP is preserved by Boolean combinations, it follows that $\varphi$ has IP.
	\end{proof}
	
	\begin{rem}
		Lemma~\ref{lem:NIP_tame} is basically equivalent to \cite[Corollary 3.2]{Ib16} (though the latter uses a slightly different language).
		There is also an analogous equivalence between stability and the so-called WAP property of a function in a dynamical system (see e.g.\ \cite{BT16}).
	\end{rem}
	
	\begin{lem}
		\label{lem:NIP_local}
		Suppose $\varphi(x,y)$ has IP on $A\times B$, where $A,B$ are $\emptyset$-type-definable. Then there are $p,q\in S(\emptyset)$ such that $p\proves A$, $q\proves B$ and $\varphi(x,y)$ has IP on $p(\fC)\times q(\fC)$.
	\end{lem}
	
	\begin{proof}
		As noticed before, we can choose $(b_n)_{n \in \omega} \subseteq B$ indiscernible and such that $\varphi(\fC,b_n) \cap A$ are independent subsets of $A$. So we can choose $a \in A$ such that $\varphi(a,b_n)$ holds if and only if $n$ is even. It is easy to check that $p := \tp(a/\emptyset)$ and $q:=\tp(b_0/\emptyset)$ satisfy our requirements.
	\end{proof}

	\begin{dfn}
		We say that $M$ is a \emph{tame model} if for some (equivalently, every) enumeration $m$ of $M$, the system $(\Aut(M),S_m(M))$ is tame.
		\xqed{\lozenge}
	\end{dfn}
	
	\begin{cor}
		\label{cor:NIP_char}
		Let $T$ be any theory. Then the following are equivalent:
		\begin{enumerate}
			\item
			\label{it:cor:NIP_char:NIP}
			$T$ has NIP.
			\item
			\label{it:cor:NIP_char:tame_fla}
			Every formula $\varphi(x,y)$ is tame.
			\item
			\label{it:cor:NIP_char:tame_vars}
			For every small model $M$ and a small tuple $x$ of variables, the dynamical system $(\Aut(M),S_{x}(M))$ is tame.
			\item
			\label{it:cor:NIP_char:tame_elts}
			For every small model $M$ and a small tuple $a$ of elements of $\fC$, the dynamical system $(\Aut(M),S_{a}(M))$ is tame.
			\item
			\label{it:cor:NIP_char:tame_models}
			Every small model of $T$ is tame.
		\end{enumerate}
		Moreover, in \eqref{it:cor:NIP_char:tame_vars}--\eqref{it:cor:NIP_char:tame_models}, we can replace ``every small model" with ``every model of cardinality $\lvert T\rvert$", and ``small tuple" with ``finite tuple".
	\end{cor}
	\begin{proof}
		The equivalence of \eqref{it:cor:NIP_char:NIP} and \eqref{it:cor:NIP_char:tame_fla} is immediate by Lemma~\ref{lem:NIP_tame}.
		
		To see that \eqref{it:cor:NIP_char:tame_fla} is equivalent to \eqref{it:cor:NIP_char:tame_vars}, note that by Corollary~\ref{cor:tame_dense}, tameness can be tested on characteristic functions of clopen sets, so tameness of $(\Aut(M),S_{x}(M))$ follows from tameness of formulas.
		
		Similarly, \eqref{it:cor:NIP_char:tame_fla} is equivalent to \eqref{it:cor:NIP_char:tame_elts}, because by Lemmas~\ref{lem:NIP_tame} and \ref{lem:NIP_local}, we can test tameness on complete types.
		
		Finally, \eqref{it:cor:NIP_char:tame_elts} trivially implies \eqref{it:cor:NIP_char:tame_models}.
		And in the other direction, if $(\Aut(M),S_a(M))$ is untame and we choose $N\succeq M$ such that $a\in N$ and $N$ is strongly $\lvert M\rvert^+$-homogeneous, then also $(\Aut(N),S_n(N))$ is untame (by Fact \ref{fct:tame_preserved}), where $n$ is an enumeration of $N$.

		For the ``moreover" part, for tuples, it is trivial (untameness is witnessed by formulas, and formulas have finitely many variables).
		For models, suppose that $T$ has IP, i.e. some formula $\varphi(x,y)$ has IP. By Lemma~\ref{lem:NIP_local}, $\varphi(x,y)$ has IP on $p(\fC) \times \fC$ for some $p \in S(\emptyset)$. Take $a \models p$. The proof of $(\leftarrow)$ in Lemma~\ref{lem:NIP_tame} easily yields a model $M$ of cardinality $|T|$, containing $a$, and such that $(\Aut(M),S_a(M))$ is untame for $a \models p$. Then, by Fact \ref{fct:tame_preserved}, the systems $(\Aut(M),S_x(M))$ and $(\Aut(M),S_m(M))$ are untame as well, where $m$ is an enumeration of $M$.
	\end{proof}

	In the $\omega$-categorical case, we obtain a simpler characterization of NIP.
	\begin{cor}
		Suppose $T$ is a countable $\omega$-categorical theory. The following are equivalent:
		\begin{itemize}
			\item
			$T$ has NIP,
			\item
			the countable model of $T$ is tame.
		\end{itemize}
		More generally, a theory $T$ is NIP if and only if it has a tame, $\aleph_0$-saturated, strongly $\aleph_0$-homogeneous model.
	\end{cor}
	\begin{proof}
		The main part is immediate by Corollary~\ref{cor:NIP_char}.
		Then implication $(\rightarrow)$ in the ``more general" case also follows from Corollary~\ref{cor:NIP_char} (and the existence of $\aleph_0$-saturated and strongly homogeneous models). In the other direction, we argue as in the ``moreover" part of Corollary~\ref{cor:NIP_char}, noticing that $\aleph_0$-saturation and strong $\aleph_0$-homogeneity of $M$ allow us to use $M$ in that argument.
	\end{proof}
	
	\begin{cor}
		If $T$ has NIP, then for every countable model $M\preceq \fC$ and countable tuple $a\in \fC$, the dynamical system $(\Aut(M),S_a(M))$ is tame, and consequently, if $T$ is countable, $E(\Aut(M),S_a(M))$ is Rosenthal.
	\end{cor}
	\begin{proof}
		Immediate by Corollary~\ref{cor:NIP_char} and Proposition~\ref{prop:dyn_BFT}.
	\end{proof}
	
	We introduce the following definition.
	\begin{dfn}
		We say that $M$ is an \emph{ambitious model} if for some (equivalently, for every) enumeration $m$ of $M$, the $\Aut(M)$-orbit of $\tp(m/M)$ is dense in $S_m(M)$ (i.e.\ $(\Aut(M),S_m(M),\tp(m/M))$ is an ambit).
		\xqed{\lozenge}
	\end{dfn}
	
	\begin{prop}
		\label{prop:amb_exist}
		Any set $A\subseteq \fC$ is contained in an ambitious model $M$ of cardinality $\lvert A\rvert+\lvert T\rvert+\aleph_0$.
	\end{prop}
	\begin{proof}
		Put $\kappa=\lvert A\rvert+\lvert T\rvert+\aleph_0$. Extend $A$ to some $M_0\preceq \fC$ of cardinality $\kappa$, enumerated by $m_0$. The weight of $S_{m_0}(M_0)$ is at most $\kappa$, so it has a dense subset of size at most $\kappa$, so we can find a group $\Sigma_0\leq \Aut(\fC)$ of size $\kappa$ such that the types over $M_0$ of elements of $\Sigma_0\cdot m_0$ form a dense subset of $S_{m_0}(M_0)$. Then we extend $\Sigma_0\cdot M_0$ to $M_1\preceq \fC$ and continue, finding an appropriate $\Sigma_1 \supseteq \Sigma_0$ and $M_2$, and so on. Then $M=\bigcup_n M_n$ satisfies the conclusion.
	\end{proof}
	
	\begin{rem}
		Alternatively, one can show that if $M$ is a model which together with some group $\Sigma$ acting on it by automorphisms satisfies $(M,\Sigma)\preceq (\fC,\Aut(\fC))$, then $M$ is ambitious, whence Proposition~\ref{prop:amb_exist} follows from the downwards Löwenheim-Skolem theorem.
		\xqed{\lozenge}
	\end{rem}

	One might ask whether we can extend Corollary~\ref{cor:NIP_char} to say that $T$ has NIP if and only if $T$ has a tame ambitious model --- we know that this is the case if $T$ is $\omega$-categorical, but the following example shows that it is not enough in general.
	
	\begin{ex}
		Suppose $M=\dcl(\emptyset)$ is a model (this is possible in an IP theory: for instance if we name all elements of a fixed model of an arbitrary IP theory).
		
		Then $S_{m}(M)$ is a singleton, so $M$ is trivially tame and ambitious.\xqed{\lozenge}
	\end{ex}
	However, any example of this sort will be G-compact, so in this case the the main result (Theorem~\ref{thm:main}) reduces to Theorem~\ref{thm:main_over_KP}, which is simpler by far to prove, and as such, not interesting from the point of view of the following analysis. This leads us to the following question.
	\begin{ques}
		\label{ques:non-g-cpct_tame}
		Is there a countable theory $T$ which is IP but not G-compact, such that some countable $M\models T$ is tame and ambitious?
		\xqed{\lozenge}
	\end{ques}

	\section{From topological dynamics to Polish spaces}
	\label{sec:top_dyn_to_Polish}
	In this section, $G$ is an abstract group and $(G,X,x_0)$ is a (compact) $G$-ambit, i.e.\ $G$ acts on $X$ by homeomorphisms and $G\cdot x_0$ is dense in $X$. In the applications, we will be mostly interested in the case where $G=\Aut(M)$, $X=S_{m}(M)$, and $x_0=\tp(m/M)$ for a suitably chosen countable model $M$ of a given countable theory $T$ and an enumeration $m$ of $M$. Another interesting case to consider is when $G=G(M)$ is a type-definable group, $X=S_G(M)$, and $x_0=\tp(e/M)$ (for a suitably chosen model $M$). However, the results of this section are completely general.
	
	We use the notation of Section~\ref{ssec:prel_topdyn} throughout. In particular, we use $EL$ for the Ellis semigroup of $G$ acting on $X$, $\cM$ for a fixed minimal left ideal in $EL$, and $u$ for a fixed idempotent in $\cM$.
	
	\subsection{Good quotients of the Ellis semigroup and the Ellis group}
	In this subsection, we find a rich Polish quotient of the Ellis group of a metric dynamical system (i.e.\ when $X$ is metrizable).
	
	We have a natural map $R\colon EL\to X$ given by $R(f)=f(x_0)$. This gives us an equivalence relation $\equiv$ on $EL$ given by $f_1\equiv f_2$ whenever $R(f_1)=R(f_2)$. Note that $R$ is continuous, so $\equiv$ is closed, and by compactness and the density of $G\cdot x_0$ in $X$, $R$ is surjective, so, abusing notation, we topologically identify $EL/{\equiv}$ with $X$. Similarly, for $A\subseteq EL$, we identify $A/{\equiv}$ with $R[A]\subseteq X$. The goal of this subsection is to find a Polish quotient of $u\cM/H(u\cM)$ which will be sufficiently well-behaved with respect to $R$.
	
	\begin{rem}
		\label{rem:commu}
		$R$ commutes with (left) multiplication in $EL$. More precisely, suppose $f_1,f_2\in EL$. Then $R(f_1f_2)=f_1(R(f_2))$. In the same way, $R$ commutes with multiplication by the elements of $G$.
	\end{rem}
	\begin{proof}
		$R(f_1f_2)=(f_1f_2)(x_0)=f_1(f_2(x_0))=f_1(R(f_2))$. From this, the second part follows, since $g\cdot f=\pi_gf$ for $g \in G$.
	\end{proof}

	Let $D=[u]_{\equiv}\cap u\cM$.

	\begin{lem}
		$D$ is a ($\tau$-)closed subgroup of $u\cM$.
	\end{lem}
	\begin{proof} Consider any $d \in \cl_\tau(D)$.
		Let $(g_i),(d_i)$ be nets as in the definition of $u\circ D$, i.e.\ such that $g_i\in G$, $g_i\to u$ and $g_id_i\to d$.
		By continuity of $R$, because $R(d_i)=R(u)$ (by the definition of $D$), and by the preceding remark, as well as left continuity of multiplication in $EL$, we have
		\[
			R(d)=\lim R(g_id_i)=\lim g_iR(d_i)=\lim g_iR(u)=R(\lim g_iu)=R(u^2)=R(u).
		\]
		This shows that $D$ is $\tau$-closed.
		
		To see that $D$ is a subgroup of $u\cM$, take any $d,d_1,d_2\in D$. Then:
		\[
		R(d_1d_2)=d_1(R(d_2))=d_1(R(u))=R(d_1u)=R(d_1)=R(u),
		\]
		\[
		R(d^{-1})=R(d^{-1}u)=d^{-1}(R(u))=d^{-1}(R(d))=R(d^{-1}d)=R(u).\qedhere
		\]
	\end{proof}
	
	The following simple example shows that the subgroups $D$ and $DH(u\cM)$ do not have to be normal in $u\cM$.
	\begin{ex}
		\label{ex:D_not_normal}
		Consider $G=S_3$ acting naturally on $X=\{1,2,3\}$ (with discrete topology), and take $x_0=1$. Then $G=u\cM$ and $D=DH(u\cM)$ is the stabilizer of $1$, which is not normal in $u\cM$. 
	\end{ex}
	
	\begin{lem}
		\label{lem:D_kernel_equiv}
		Let $f_1,f_2\in u\cM$. Then $f_1\equiv f_2$ (i.e.\ $R(f_1):=f_1(x_0)=f_2(x_0)=:R(f_2)$) if and only if $f_1^{-1}f_2\in D$ (note that here, $f_1^{-1}$ is the inverse of $f_1$ in $u\cM$, not the inverse function), i.e.\ $f_1D=f_2D$. (And thus $u\cM/{\equiv}$ and $u\cM/D$ can and will be identified as sets.)
	\end{lem}
	\begin{proof}
		In one direction, if $f_1\equiv f_2$,
		\[
			R(f_1^{-1}f_2)=f_1^{-1}(R(f_2))=f_1^{-1}(R(f_1))=R(f_1^{-1}f_1)=R(u).
		\]
		In the other direction, if $R(f_1^{-1}f_2)=R(u)$, then
		\[
			R(f_1)=R(f_1u)=f_1(R(u))=f_1(R(f_1^{-1}f_2))=R(f_1f_1^{-1}f_2)=R(f_2)\qedhere
		\]
	\end{proof}
	
	By Fact~\ref{fct:tau_top}, we have the compact Hausdorff topological group $u\cM/H(u\cM)$. Since $D$ is closed in $u\cM$ (and hence compact), it follows that $H(u\cM)D/H(u\cM)$ is a closed subgroup in the quotient. Consequently, $u\cM/(H(u\cM)D)$ (which can also be described as $(u\cM/H(u\cM))/(DH(u\cM)/H(u\cM))$) is a compact Hausdorff space (by Fact~\ref{fct:quotient_by_closed_subgroup}). By Lemma~\ref{lem:D_kernel_equiv}, the quotient map $u\cM\to u\cM/(H(u\cM)D)$ factors through $u\cM/{\equiv}$, which we identify with $R[u\cM]\subseteq X$, giving us a commutative diagram
	\begin{center}
		\begin{tikzcd}
			u\cM\arrow[r] \arrow[d,"R"]& u\cM/H(u\cM)\arrow[d]\\
			R[u\cM] \arrow[r,"\widehat j"] & u\cM/(H(u\cM)D).
		\end{tikzcd}
	\end{center}
	
	\begin{rem}
		Suppose $\sim$ is a closed equivalence relation on a compact Hausdorff space $X$, while $F\subseteq X$ is closed. Then the set $[F]_\sim$ of all elements equivalent to some element of $F$ is also closed.
	\end{rem}
	\begin{proof}
		$[F]_\sim$ is the projection of $(X\times F)\cap {\sim}$ onto the first axis.
	\end{proof}
	
	\begin{lem}
		\label{lem:jhat_cont}
		On $u\cM/{\equiv}=u\cM/D$, the topology induced from the $\tau$-topology on $u\cM$ is refined by the subspace topology inherited from $EL/{\equiv}=X$.
		
		Consequently, $\widehat j$ in the above diagram is continuous (with respect to the quotient $\tau$ topology on $u\cM/H(u\cM)D$.)
	\end{lem}
	\begin{proof}
		We need to show that if $F\subseteq u\cM$ is $\tau$-closed and right $D$-invariant (i.e. $FD=F$), then there is a closed $\equiv$-invariant $\widetilde F\subseteq EL$ such that $\widetilde F\cap u\cM=F$. By the preceding remark, since $\equiv$ is closed, it is enough to check that $[\bar F]_{\equiv}\cap u\cM=F$, where $\bar F$ is the closure of $F$ in $EL$.
		
		Let $f'\in [\bar F]_{\equiv}\cap u\cM$. Then we have a net $(f_i)\subseteq F$ such that $f_i\to f$ and $f\equiv f'$. By Fact~\ref{fct:tau_top}\eqref{it:lem:tau_top:nearlycont}, in this case, $f_i$ converges in the $\tau$-topology to $uf$, which is an element of $F$ (because $F$ is $\tau$-closed). Since $F$ is right $D$-invariant (and hence $\equiv$-invariant in $u\cM$), it is enough to show that $f'\equiv uf$. But this is clear since
		\[
			R(uf)=u(R(f))=u(R(f'))=R(uf')=R(f').\qedhere
		\]
	\end{proof}
	
	As indicated before, we want to find diagrams similar to the one used in Lemma~\ref{lem:easy} for use in the proof of \hyperref[mainthm]{Main Theorem}. As an intermediate step, we would like to complete the following diagram.
	
	\begin{center}
		\begin{tikzcd}
			EL\arrow[d]\arrow[r,swap,outer sep=3pt,"f\mapsto fu"] & \cM\arrow[d]\arrow[r,swap,outer sep=3pt,"f\mapsto uf"] & u\cM\arrow[d] \\
			EL/{\equiv} \arrow[r,dashed]\arrow[d,equal] & \cM/{\equiv} \arrow[r,dashed]\arrow[d,equal] & u\cM/D\\[-1em]
			X& R[\cM]& {}
		\end{tikzcd}
	\end{center}
	The dashed arrow on the right exists: if $R(f_1)=R(f_2)$, then $u(R(f_1))=u(R(f_2))$, so, by Remark~\ref{rem:commu}, also $R(uf_1)=R(uf_2)$, and hence $uf_1D = uf_2D$ by Lemma \ref{lem:D_kernel_equiv}. Unfortunately, there is no reason for the arrow on the left to exist (i.e.\ $f_1\equiv f_2$ does not necessarily imply $f_1u\equiv f_2u$). However, we can remedy it by replacing $EL/{\equiv}$ with $EL/{\equiv'}$, where $\equiv'$ is given by $f_1\equiv' f_2$ iff $R(f_1)=R(f_2)$ and $R(f_1u)=R(f_2u)$. This gives us a commutative diagram, substituting for the above one:
	\begin{center}
		\begin{tikzcd}
			EL\arrow[d]\arrow[r,swap,outer sep=3pt,"f\mapsto fu"] & \cM\arrow[d]\arrow[r,swap,outer sep=3pt,"f\mapsto uf"] & u\cM\arrow[d] \\
			EL/{\equiv'} \arrow[r]\arrow[d] & \cM/{\equiv} \arrow[r]\arrow[d,equal] & u\cM/D\\[-1em]
			X& R[\cM]& {}
		\end{tikzcd}
	\end{center}
	
	\begin{prop}\label{prop: quotients of EL are Polish}
		$EL/{\equiv}$ and $EL/{\equiv'}$ are both compact Hausdorff spaces. 
		
		If $X$ is second-countable (by compactness, equivalently, Polish), so is $EL/{\equiv}$, as well as $EL/{\equiv'}$.
	\end{prop}
	\begin{proof}
		Since $EL/{\equiv}$ is homeomorphic to $X$, the part concerning $EL/{\equiv}$ is clear.
		
		For $EL/{\equiv'}$, note first that $\cM/{\equiv}$ is a closed subspace of $EL/{\equiv}$, and hence it is Polish whenever $X$ is. To complete the proof, use compactness of $EL$, Hausdorffness of $EL/{\equiv}$ and $\cM/{\equiv}$, and continuity of the diagonal map $d \colon EL\to EL/{\equiv}\times \cM/{\equiv}$ given by $f\mapsto ([f]_\equiv,[fu]_\equiv)$ in order to deduce that $EL/{\equiv'}$ is homeomorphic to $d[EL]$ which is closed.
	\end{proof}

	\begin{prop}
		\label{prop:uM/HuMD_Polish}
		If $X$ is metrizable, then $u\cM/H(u\cM)D$ is a Polish space.
	\end{prop}
	\begin{proof} The following diagram of maps is essential and explained below.
		\begin{center}
			\begin{tikzcd}
				&[-1.5em] \overline{u\cM} \ar{r}\ar{d} & u\cM/H(u\cM)\ar{d} \\
				R[\overline{u\cM}]\arrow[r,equals]&\overline{u\cM}/{\equiv}\ar{r} & u\cM/H(u\cM)D
			\end{tikzcd}
		\end{center}
		
		Note that $\overline{u\cM}$ is a compact space (equipped with the subspace topology from $EL$). Consequently, $R[\overline{u\cM}]=\overline{u\cM}/{\equiv}$ is a compact Polish space. The top arrow, given by $f\mapsto uf/H(u\cM)$, is continuous by Fact~\ref{fct:strange_cont}. The induced map $\overline{u\cM}\to u\cM/H(u\cM)D$ factors through the quotient map $\overline{u\cM} \to \overline{u\cM}/{\equiv}$ yielding a continuous map $\overline{u\cM}/{\equiv}\to u\cM/H(u\cM)D$ (if $f_1\equiv f_2$, then $uf_1\equiv uf_2$, and then we apply Lemma~\ref{lem:D_kernel_equiv}). Hence, $u\cM/H(u\cM)D$ is a compact Hausdorff space which is a continuous image of a compact Polish space. As such, it must be Polish by Fact \ref{fct: preservation of metrizability}.
	\end{proof}

	\subsection{Tameness and Borel ``retractions"}
	
	\begin{prop}
		\label{prop:last_borel}
		Suppose $(G,X)$ is a tame metric system. Then for any $f_0\in EL$ the map $f\mapsto f_0f$ is $\equiv$-preserving and the induced transformation of $EL/{\equiv}$ is Borel.
		
		In particular, the map $\cM/{\equiv}\to u\cM/{\equiv}$ induced by $p\mapsto up$ is Borel, where both spaces are equipped with the subspace topology from $X=EL/{\equiv}$.
	\end{prop}
	\begin{proof}
		Preserving $\equiv$ follows immediately from Remark \ref{rem:commu}. 
		The induced transformation of $EL/{\equiv}$ is the same as simply $f_0$ once we identify $X$ with $EL/{\equiv}$, and $f_0$ is Borel by Fact~\ref{fct:tame_borel}.
	\end{proof}
	
	\begin{cor}
		\label{cor:borel_map}
		The map $\cM/{\equiv}\to u\cM/H(u\cM)D$, given by $[f]_\equiv\mapsto uf/H(u\cM)D$, is Borel, where the former is equipped with subspace topology from $EL$, while the latter has topology induced from the $\tau$ topology.
		
		Similarly, the map $EL/{\equiv'}\to u\cM/H(u\cM)D$, given $[f]_{\equiv'}\mapsto ufu/H(u\cM)D$, is Borel.
	\end{cor}
	\begin{proof}
		The first map is the composition of the continuous map $\hat j\colon u\cM/{\equiv}\to u\cM/H(u\cM)D$ from Lemma~\ref{lem:jhat_cont} and the Borel function from the second part of Proposition~\ref{prop:last_borel}. The second map is the composition of the first one with the continuous map $[f]_{\equiv'}\mapsto [fu]_{\equiv}$.
	\end{proof}

	\subsection{Polish group quotients of the Ellis group}
	By Proposition \ref{prop:uM/HuMD_Polish}, we already know that for metric dynamical systems, the quotient $u\cM/H(u\cM)D$ is a Polish space. However, we want to obtain a Polish group, and as we have seen in Example~\ref{ex:D_not_normal}, $DH(u\cM)$ may not be normal, so we need to slightly refine our approach.
	\begin{prop}
		Suppose $G$ is a compact Hausdorff topological group and $H\leq G$ is such that $G/H$ is metrizable. Then $G/\Core(H)$ is a compact Polish group (where $\Core(H)$ is the normal core of $H$ in $G$, i.e. the intersection of all its conjugates).
	\end{prop}
	\begin{proof}
		Let $\varphi \colon G\to \Homeo(G/H)$ be the homomorphism defined by $\varphi(g)(aH) :=gaH$, where $\Homeo(G/H)$ is the group of all homeomorphisms of $G/H$. 
		
		Recall that a compact Hausdorff space possesses a unique uniformity inducing the given topology.
		The action of $G$ on $G/H$ by left translations is continuous, so also uniformly continuous (as $G$ and $G/H$ are compact Hausdorff). Therefore, if $(g_i)_i$ is a convergent net, then $(g_i\cdot gH)_i$ converges uniformly in $gH \in G/H$. This yields continuity of $\varphi$ with respect to the uniform convergence topology on $\Homeo(G/H)$.
		
		It is easy to check that $\ker(\varphi)=\Core(H)$, and since $G/H$ is a compact Polish space, $\Homeo(G/H)$ is a Polish group by \cite[9.B(8)]{Kec95}. By compactness of $G$, it follows that $\varphi[G]$ is a Polish group, and hence --- by Remark \ref{rem: continuous surjection is closed} --- so is $G/\Core(H)$.
	\end{proof}
	
	\begin{cor}\label{cor:Polish_quotient_Core(D)}
		For a metric dynamical system, $u\cM/H(u\cM)\Core(D)$ is a compact Polish group.
	\end{cor}
	\begin{proof}
		Immediate by the preceding proposition, as $u\cM/H(u\cM)$ is a compact Hausdorff group and $u\cM/H(u\cM)D$ is a compact Polish space (by Proposition~\ref{prop:uM/HuMD_Polish}).
	\end{proof}
	
	In the case of \emph{tame} metric dynamical systems, the situation is a little cleaner. Namely, we will show that $u\cM/H(u\cM)$ itself is already Polish.
	
	\begin{dfn}
		A topological space $X$ has {\em countable tightness} if for every $A\subseteq X$ and every $x\in \overline A$, there is a countable set $B\subseteq A$ such that $x\in \overline B$.
		\xqed{\lozenge}
	\end{dfn}
	
	\begin{fct}[Engelking]\label{fct:tightness}
		A compact (Hausdorff) topological group of countable tightness is metrizable.
	\end{fct}
	\begin{proof}
		\cite[Corollary 4.2.2]{AT08}.
	\end{proof}
	
	\begin{prop}
		\label{prop:closed_image_is_tight}
		The image of a countably tight space via a closed continuous map is countably tight.
	\end{prop}
	\begin{proof}
		Let $X$ be a countably tight space, and let $f\colon X\to Y$ be a closed and continuous surjection. Choose an arbitrary $A\subseteq Y$ and $y\in \overline A$. Note that since $f$ is closed and onto, we have that $\overline A\subseteq f\left[\overline{f^{-1}[A]}\right]$, so there is some $x\in \overline{f^{-1}[A]}$ such that $f(x)=y$. Choose $B'\subseteq f^{-1}[A]$ countable such that $x\in \overline {B'}$, and let $B=f[B']$. Since $f$ is continuous, $f^{-1}\left[\overline B\right]\supseteq \overline{B'}$, so in particular, $x\in f^{-1}\left[\overline B\right]$, so $y\in \overline B$.
	\end{proof}
	
	\begin{prop}\label{prop:NIP gives metrizability}
		If $(G,X)$ is a tame metric dynamical system, then the group $u\cM/H(u\cM)$ is metrizable (and hence a Polish group).
	\end{prop}
	\begin{proof}
		Note that if $(G,X)$ is tame, then, by Proposition~\ref{prop:dyn_BFT}, $\overline{u\cM}\subseteq EL$ is a Rosenthal compactum, so --- via the Fréchet-Urysohn property we have by Fact~\ref{fct: Rosnthal implies Frechet} --- it is countably tight. Furthermore, by Fact \ref{fct:strange_cont}, the function $f\mapsto uf/H(u\cM)$ defines a continuous surjection from $\overline{u\cM}$ to $u\cM/H(u\cM)$, and hence a continuous closed mapping. Hence, the result follows by Proposition~\ref{prop:closed_image_is_tight} and Fact~\ref{fct:tightness}.
	\end{proof}

	\section{The main theorem}
	\label{sec:main_thm}
	In this section, we assume that $T$ is countable, and we fix a countable ambitious $M\models T$ enumerated by $m$. Note that $(\Aut(M),S_m(M),\tp(m/M))$ is an $\Aut(M)$-ambit, so the results of Section~\ref{sec:top_dyn_to_Polish} apply. As before, we denote by $EL$ the Ellis semigroup of $(\Aut(M),S_m(M))$, and we fix a minimal left ideal $\cM\unlhd EL$ and an idempotent $u\in \cM$. We also use the notation of Section~\ref{sec:top_dyn_to_Polish}, so in particular for $f\in EL$, $R(f)=f(\tp(m/M))$. 
	
	\subsection{\texorpdfstring{Topological dynamics for $\Aut(M)$}{Topological dynamics for Aut(M)}}\label{subsection: adaptation from KPR15}

	This subsection is an adaptation of topological dynamics developed for the group $\Aut(\fC)$ in \cite{KPR15} to the context of $\Aut(M)$. Many of the arguments used in \cite{KPR15} can be translated almost immediately. However, we rephrase some of them here, mostly to make it easier to see how the same principles can be applied in wider contexts.
	
	Notice that $\Aut(\fC)$ does not act on $S_m(M)$, and since $M$ may be neither saturated nor homogeneous, we cannot use any compactness argument directly for orbits of $\Aut(M)$. Instead, we have a surrogate in the form of the following lemma, which we will proceed to use to derive other interesting properties.
	\begin{lem}[pseudocompactness]
		\label{lem:pseudocompactness}
		Whenever $(\sigma_i)_i$ and $(p_i)_i$ are nets in $\Aut(M)$ and $S_m(M)$ (respectively) such that $\tp(\sigma_i(m)/M)\to q_1$, $p_i\to q_2$ and $\sigma_i(p_i)\to q_3$ for some $q_1,q_2,q_3\in S_m(M)$, there are $\sigma'_1,\sigma'_2\in \Aut(\fC)$ such that $\tp(\sigma'_1(m)/M)=q_1$, $\tp(\sigma'_2(m)/M)=q_2$ and $\tp(\sigma'_1\sigma'_2(m)/M)=q_3$.
	\end{lem}
	\begin{proof}
		Let $n_i\models p_i$, and $\bar \sigma_i$ be an extension of $\sigma_i$ to an automorphism of $\fC$. Then, by the assumptions, for every $\varphi_1(x),\varphi_2(x),\varphi_3(x)$ in $q_1,q_2$ and $q_3$ (resp.) we have, for sufficiently large $i$, $\models \varphi_1(\bar\sigma_i(m))\land \varphi_2(n_i)\land \varphi_3(\bar \sigma_i(n_i))\land mn_i\equiv \bar \sigma_i(m)\bar \sigma_i(n_i)$. So, by compactness, we get $m_1,m_2,m_3$ such that $\models q_1(m_1)\land q_2(m_2)\land q_3(m_3)\land mm_2\equiv m_1m_3$. Any $\sigma'_1, \sigma'_2$ such that $\sigma'_2(m)=m_2$, $\sigma'_1(mm_2)=m_1m_3$ satisfy the conclusion of the lemma.
	\end{proof}
	
	\begin{prop}
		\label{prop:twopoints}
		For every $f\in EL$ and $p\in S_m(M)$, there are $\sigma'_1, \sigma'_2\in \Aut(\fC)$ such that $\tp(\sigma'_1(m)/M)=R(f)$, $\tp(\sigma'_2(m)/M)=p$ and $\tp(\sigma'_1\sigma'_2(m)/M)=f(p)$.
	\end{prop}
	\begin{proof}
		Apply Lemma~\ref{lem:pseudocompactness} to any net $(\sigma_i)_i$ convergent to $f$ and $p_i:=p$ constant.
	\end{proof}
	
	Denote by $r$ the map $EL\to \Gal(T)$ defined as the composition of $R$ and the natural map $S_m(M)\to \Gal(T)$ from Fact~\ref{fct:sm_to_gal}, i.e.\ $r(f)=[f(\tp(m/M))]_{\equiv_L^M}$ (here, $S_m(M)/{\equiv_L^M}$ is naturally identified with $\Gal(T)$).
	
	\begin{lem}
		\label{lem:homom}
		$r$ is a semigroup epimorphism.
	\end{lem}
	\begin{proof}
		First, we show that $r$ is a homomorphism.
		Take any $f_1,f_2\in EL$. By Proposition~\ref{prop:twopoints} (applied to $f:=f_1$ and $p:=R(f_2)$), we have $\sigma'_1,\sigma'_2\in \Aut(\fC)$ such that $\tp(\sigma'_1(m)/M)=R(f_1)$, $\tp(\sigma'_2(m)/M)=R(f_2)$, and $\tp(\sigma'_1\sigma'_2(m)/M)=f_1(R(f_2))$.
		Then $r(f_i)=\sigma'_i\Autf_L(\fC)$ and
		\[
			r(f_1f_2)=[f_1R(f_2)]_{\equiv_L^M}=[\tp(\sigma_1'\sigma_2'(m)/M)]_{\equiv_L^M}=\sigma_1'\sigma_2'\Autf_L(\fC)=r(f_1)r(f_2).
		\]

		It remains to check that $r$ is onto. Consider any $ \sigma' \in \Aut(\fC)$. Since $M$ is ambitious, we can find a net $(\sigma_i)_i$ in $\Aut(M)$ such that $\sigma_i(\tp(m/M)) \to \tp(\sigma'(m)/M)$. By compactness of $EL$, we can assume that $(\sigma_i)_i$ converges to some $f \in EL$. Then $R(f) = f(\tp(m/M)) = \lim \sigma_i(\tp(m/M)) = \tp(\sigma'(m)/M)$, so $r(f)=\sigma'\Autf_L(\fC)$.
	\end{proof}

	The next proposition is a counterpart of \cite[Theorem 2.8(2)]{KPR15} whose proof is a straightforward adaptation of the proof from \cite{KPR15} using Lemma \ref{lem:pseudocompactness}, so we will only say a few words about the proof. Similarly, we could adapt the proof of \cite[Theorem 2.8(1)]{KPR15} to get continuity of $r\restr_{u\cM}$, but in this paper, we give another proof of continuity (see Proposition \ref{lem:r_restr_to_top_quot}).
	
	\begin{prop}
		\label{prop:H(uM)_in_ker}
		$H(u\cM)\leq \ker r$. In other words, for all $f\in H(u\cM)$ we have $f(\tp(m/M))\equiv_L^M\tp(m/M)$.
	\end{prop}
	\begin{proof}
		One should follow the lines of the proof of \cite[Theorem 2.7(2)]{KPR15}, replacing $\bar c, \fC, \fC', \pi_0$ by $m, M, \fC,R$ (respectively) and using Lemma~\ref{lem:pseudocompactness} in appropriate places.
		
		Briefly, if we denote by $\tilde G_n$ the set of elements $f\in EL$ such that any realization of $R(f)\in S_m(M)$ is at Lascar distance at most $n$ from $m$ (the tuple enumerating the fixed countable model $M$), then $\ker r=\bigcup \tilde G_n$, so $u\in \tilde G_n$ for some $n$. Then, we find a larger $n'$ such that for every element $f\in H(u\cM)$ we have $f\in \tilde G_{n'}$.
	\end{proof}
	
	\begin{prop}
		\label{prop:homom}
		$r\restr_{u\cM}\colon u\cM\to \Gal(T)$ is a surjective group homomorphism.
	\end{prop}
	\begin{proof}
		Since $u\cM=uELu$ and $\Gal(T)$ is a group, it follows from Proposition \ref{lem:homom} that $r[u\cM]=r(u)r[EL]r(u)=r(u)\Gal(T)r(u)=\Gal(T)$.
	\end{proof}

	\begin{prop}
		\label{prop:restr_quot}
		$r\restr_{\cM}\colon \cM\to \Gal(T)$ is a topological quotient map.
	\end{prop}
	\begin{proof}
		By Remark \ref{rem: continuous surjection is closed}, $R\colon EL\to S_m(M)$ is a quotient map, and thus so is $r$ (as the composition of $R$ and the quotient map $S_m(M)\to \Gal(T)$).
		By Remark \ref{rem: continuous surjection is closed}, we also have that the map $\theta\colon EL \to \cM$ given by $f\mapsto fu$ is a quotient map. Moreover, $r(f)=r(fu)$ for all $f \in EL$. 

		In conclusion, we have the following commutative diagram, where the top and left arrows are quotient mappings discussed above. From this, it follows that $r\restr_{\cM}$ is also a quotient map (and, in fact, a factor of $r$).

		\begin{center}
			\begin{tikzcd}
				EL \arrow[rr,"r"]\arrow[dr,"\theta"]& &\Gal(T)\\
				& \cM\arrow[ur,"r\restr_{\cM}"] &
			\end{tikzcd}
		\end{center}
	\end{proof}

	\begin{prop}[Corresponding to {\cite[Lemma 4.7]{KP17}} and {\cite[Lemma 2.11]{KPR15}}]
		\label{prop:id_clsd}
		Denote by $J$ the set of all idempotents in $\cM$. Then $\overline{J}\subseteq \ker r\cap \cM$.
	\end{prop}
	\begin{proof}
		Put 
		\[
		F:=\{ \tp(\alpha/M) \mid (\exists \alpha_1,\alpha_2)(\alpha_1 \equiv_M \alpha_2 \equiv_M \alpha \wedge m\alpha_2 \equiv \alpha_1\alpha)\}.
		\]
		$F$ is clearly a closed subset of $S_m(M)$, so $R^{-1}[F]$ is closed in $EL$. Thus, it remains to check that $J \subseteq R^{-1}[F] \subseteq \ker r$.
		
		For the first inclusion, consider any $v \in J$. Take $\sigma_1',\sigma_2'$ according to Proposition~\ref{prop:twopoints} for $f:=v$ and $p:=R(v)$. Then
		\[
			\tp(\sigma_1'\sigma_2'(m)/M) = vR(v)=R(v^2)=R(v)=\tp(\sigma_1'(m)/M)=\tp(\sigma_2'(m)/M). 
		\]
		So $\alpha := \sigma_1'\sigma_2'(m)$, $\alpha_1 :=\sigma_1'(m)$, and $\alpha_2:= \sigma_2'(m)$ witness that $R(v) \in F$. 
		
		To see the second inclusion, consider any $\tp(\alpha/M) \in F$. Take $\alpha_1,\alpha_2$ witnessing this. Take $\sigma_1'$ mapping $m\alpha_2$ to $\alpha_1\alpha$. Since we easily see that $\alpha_2 \equiv m$, we can choose $\sigma_2'$ mapping $m$ to $\alpha_2$. Then $\tp(\sigma_1'(m)/M)=\tp(\sigma_2'(m)/M) = \tp(\sigma_1'\sigma_2'(m)/M)$. This implies that $\sigma_1',\sigma_2' \in \Autf_L(\fC)$, so $\tp(\alpha/M) =\tp(\sigma_1'\sigma_2'(m)/M) \equiv_L^M \tp(m/M)$.
\end{proof}

	\begin{prop}
		\label{lem:r_restr_to_top_quot}
		$r\restr_{u\cM}$ is a topological group quotient map (where $u\cM$ is equipped with the $\tau$ topology).
	\end{prop}
	\begin{proof}
		In light of Proposition~\ref{prop:homom}, it is enough to show that $r\restr_{u\cM}$ is a topological quotient map.
		
		For continuity, note that if $F\subseteq \Gal(T)$ is closed, then $F':=r^{-1}[F]\cap \overline{u\cM}$ is closed in $\overline{u\cM}$ by continuity of $r$. %Proposition \ref{prop:restr_quot}
		From Fact~\ref{fct:strange_cont}, it follows that $uF'/H(u\cM)$ is closed. But because $u\in\ker r$, we get $uF'\subseteq F'$, so $uF'=F'\cap u\cM$ is $(\ker r\cap u\cM$)-invariant, and hence also $H(u\cM)$-invariant by Proposition \ref{prop:H(uM)_in_ker}. It follows that $uF'=r^{-1}[F]\cap u\cM(=r\restr_{u\cM}^{-1}[F])$ is $\tau$-closed, so $r\restr_{u\cM}$ is continuous.

		Let $P_u:=\ker r\cap u\cM(=\ker (r\restr_{u\cM}))$ and $S:=u(u\circ P_u)=\cl_\tau(P_u)$. We will need the following claim.
		
		\begin{clm*}
			$r^{-1}[r[S]]\cap \cM$ is closed.
		\end{clm*}
		By the claim, $r^{-1}[r[S]]\cap \cM$ is closed in $\cM$, so by Proposition~\ref{prop:restr_quot}, $r[S]$ is a closed subset of $\Gal(T)$. In particular, it must contain the closure of the identity in $\Gal(T)$, i.e. $\Gal_0(T)$. On the other hand, by continuity of $r\restr_{u\cM}$, the preimage of $\Gal_0(T)$ by $r\restr_{u\cM}$ is a $\tau$-closed set containing $P_u$, and thus also $S$. It follows that $r[S]=\Gal_0(T)$.
		
		Now, given any $F\subseteq \Gal(T)$, if $r\restr_{u\cM}^{-1}[F]=r^{-1}[F]\cap u\cM$ is $\tau$-closed, it is compact, so $F/\Gal_0(T)$ is closed in $\Gal_{KP}(T)$ (by continuity of $r\restr_{u\cM}$). Furthermore, if $r\restr_{u\cM}^{-1}[F]$ is $\tau$-closed, it is also $S$-invariant (because it is $P_u$-invariant and the group operation on $u\cM$ is separately continuous in the $\tau$-topology), so it follows from the last paragraph that $F=F\cdot \Gal_0(T)$, and hence $F$ is closed. Thus, we only need to prove the claim.
		
		\begin{clmproof}[Proof of the claim]
			First, we present Lemmas 4.1 and 4.4-4.6 of \cite{KP17} in a concise form. Let $J$ be the set of idempotents in $\cM$.
			\begin{itemize}
				\item
				For any any $v,w \in J$, $vP_w=P_v$. Indeed, $v,w\in \ker r\restr_{\cM}$, so $vP_w\subseteq P_v$ and $wP_v\subseteq P_w$. Hence, $P_v=vP_v=vwP_v\subseteq vP_w\subseteq P_v$, and so $vP_w=P_v$.
				\item
				$S=S\cdot P_u$. This is true, because by left continuity of multiplication in $EL$, for any $f\in P_u$, we have $Sf=u(u\circ P_u) f=u(u\circ (P_uf))$ and clearly $P_uf=P_u$.
				\item
				Since $P_u=\ker (r\restr_{u\cM})$, it follows immediately from the preceding point that $S=r^{-1}[r[S]]\cap u\cM$.
				\item
				$r^{-1}[r[S]]\cap \cM=J\cdot S$. The inclusion $(\supseteq)$ follows from the fact that $J\subseteq \ker r$. To show the opposite inclusion, take any $ f\in r^{-1}[r[S]]\cap \cM$; then $f\in r^{-1}[r[S]]\cap v\cM$ for some $v\in J$. We see that $r(uf)=r(f)\in r[S]$, so, by the preceding point, $uf\in S$, and thus $f=vf=vuf \in vS \subseteq J \cdot S$.
				\item
				$r^{-1}[r[S]]\cap \cM=\bigcup_{v \in J} v\circ P_u$. To show this, we will use basic properties of $\circ$. The inclusion $(\subseteq)$ follows from the preceding point and the observation that 
				for any $v \in J$ we have $vS=vu(u \circ P_u) \subseteq vu^2 \circ P_u = v \circ P_u$.
				For the opposite inclusion, note that $r[v\circ P_u]=r[u(v\circ P_u)] \subseteq r[uv \circ P_u] = r[u \circ P_u]=r[u(u \circ P_u)]=r[S]$. 
			\end{itemize}
			So we only need to show that ${\bigcup_v v\circ P_u}$ is closed in $\cM$. This is \cite[Lemma 4.8]{KP17}, but we repeat the proof.		
		
			Let $f\in \overline{\bigcup_v v\circ P_u}$. Then we have nets $(v_i)_i$ in $J$ and $(f_i)_i$ in $\cM$ such that $f_i\in v_i\circ P_u$ and $f_i\to f$. By compactness, we can assume without loss of generality that the net $(v_i)$ converges to some $v\in \overline J$. Then,
			by considering neighborhoods of $v$ and $f$, we can find nets $(\sigma_j)_j$ in $\Aut(M)$ and $(p_j)_j$ in $P_u$ such that $\sigma_j\to v$ and $\sigma_jp_j\to f$, so $f\in v\circ P_u$. By Proposition~\ref{prop:id_clsd}, $v\in \ker r\cap \cM$, and as such, $v\in P_w=wP_u$ for some $w\in J$ (where the last equality follows from the first bullet above). So $v=wp$ for some $p\in P_u$. Furthermore, $P_u$ is a group (as the kernel of a group homomorphism), so
			\[
			f\in v\circ P_u=v\circ (p^{-1}P_u)\subseteq v\circ (p^{-1}\circ P_u)\subseteq (vp^{-1})\circ P_u=w\circ P_u,
			\]
			and we are done.
		\end{clmproof}
		The proof of the proposition is complete. 
	\end{proof}

	Recall that $D$ is the $\tau$-closed subgroup of $u\cM$ consisting of all $f\in u\cM$ with $R(f)=R(u)$; $\Core(D)$ is the intersection of all conjugates of $D$ in $u\cM$.
	\begin{cor}
		\label{cor:polish_quotient}
		The homomorphism $r$ induces a topological group quotient mapping $\hat r\colon u\cM/H(u\cM)\Core(D)\to \Gal(T)$, as well as another topological quotient mapping $u\cM/H(u\cM)D\to \Gal(T)$.
	\end{cor}
	\begin{proof}
		By Proposition \ref{prop:H(uM)_in_ker}, we know that $H(u\cM)\leq \ker(r)$. On the other hand, since $u\in \ker r$, we have $D\leq \ker r$, and thus also $\Core(D)\leq \ker r$. So $r$ factors through both $u\cM/H(u\cM)D$ and $u\cM/H(u\cM)\Core(D)$. The fact that the resulting factors of $r$ are topological quotient mappings follows from Proposition \ref{lem:r_restr_to_top_quot}.
	\end{proof}

	\subsection{Lemmas: Borel cardinality}
	The following two lemmas follow from the existence of Borel sections of continuous surjections between compact Polish spaces (i.e. Fact \ref{fct:borel_section}). The arguments are similar to the proof of Lemma \ref{lem:easy}, so we will skip them.
	\begin{lem}
		\label{lem:main_1}
		Suppose we have a commutative diagram
		\begin{center}
			\begin{tikzcd}
				B\ar[d]\ar[r, two heads]&G'\ar[d] \\
				C\ar[r]&Q 
			\end{tikzcd}
		\end{center}
		
		where:
		\begin{itemize}
			\item
			$B$, $C$ and $G'$ are compact Polish spaces,
			\item
			the maps $B \to C$ and $B \to G'$ are continuous,
			\item
			the map $B \to G'$ is surjective.
		\end{itemize}
		Denote by $E_C$ and $E_{G'}$ the equivalence relations on $C$ and $G'$ (respectively) induced from equality on $Q$. Then:
		\begin{enumerate}
			\item
			if $E_C$ is closed [resp. Borel, or analytic, or $F_\sigma$, or clopen], so is $E_{G'}$,
			\item
			$E_{G'}\leq_B E_C$.
		\end{enumerate}
	\xqed{\lozenge}
	\end{lem}

	\begin{lem}
		\label{lem:main_2}
		Suppose we have a commutative diagram
		\begin{center}
			\begin{tikzcd}
				A\ar[d, two heads]\ar[r, "\mathrm{Borel}"]&G'\ar[d]\\
				C\ar[r]&Q
			\end{tikzcd}
		\end{center}
		
		where:
		\begin{itemize}
			\item
			$A$, $C$ and $G'$ are compact Polish spaces,
			\item
			$A\to G'$ is Borel, while $A\to C$ is continuous,
			\item
			$A\to C$ is surjective.
		\end{itemize}
		Then $E_C\leq_B E_{G'}$, where $E_C$ and $E_{G'}$ are defined as in Lemma~\ref{lem:main_1}.
		\xqed{\lozenge}
	\end{lem}
	
	The following lemma shows that every strong type $E$ defined on the set of realizations of an arbitrary complete type over $\emptyset$ can be considered (essentially, for the purposes of the Main Theorem) as being defined on $[m]_{\equiv}$, where $m$ enumerates an arbitrary countable model.
	
	\begin{lem}
		\label{lem:every_stype_on_m}
		Suppose $T$ is countable.	
		Assume that $E$ is a strong type defined on $p(\fC)$ for $p=\tp(a/\emptyset)$ for some countable tuple $a$, while $M$ is an arbitrary countable model, enumerated by $m$.
		
		Then there is a strong type $E'$ on $[m]_{\equiv}$ such that:
		\begin{itemize}
			\item
			$E$ is type-definable [resp. Borel, or analytic, or $F_\sigma$, or relatively definable] if and only if $E'$ is,
			\item
			there are Borel maps $r_1\colon S_m(M)\to S_a(M)$ and $r_2 \colon S_a(M)\to S_m(M)$ such that $r_1$ and $r_2$ are Borel reductions between $(E')^M$ and $E^M$ (in particular, $E'\sim_B E$), satisfying $r_1(\tp(m/M))=\tp(a/M)$ and $r_2(\tp(a/M))=\tp(m/M)$, and
			\item 
			the induced maps $r_1'\colon [m]_{\equiv}/E' \to p(\fC)/E$ and $r_2'\colon p(\fC)/E \to [m]_{\equiv}/E'$ are $\Gal(T)$-equivariant homeomorphisms, and $r_2'$ is the inverse of $r_1'$.
		\end{itemize}
The maps $r_1'$ and $r_2'$ satisfying the last two items are uniquely determined  by $r_1'([\sigma(m)]_{E'})=[\sigma(a)]_{E}$ and $r_2'([\sigma(a)]_{E})=[\sigma(m)]_{E'}$ for all $\sigma \in \Aut(\fC)$.
	\end{lem}
	\begin{proof}
		Let $N\succeq M$ be a countable model containing $a$, and enumerate it by $n\supseteq am$.
		
		Then we have the restriction maps $S_n(N)\to S_a(M)$, $S_n(N)\to S_m(M)$, which fit in the commutative diagram:
		\begin{center}
			\begin{tikzcd}
			S_m(M) \ar[dr, two heads] & S_n(N) \ar[l, two heads] \ar[r, two heads] \ar[d, two heads] & S_a(M) \ar[d, two heads] \\
			&\Gal(T) \ar[r, two heads] & p(\fC)/E.
			\end{tikzcd}
		\end{center}
		In this diagram, the maps to $\Gal(T)$ are given by Fact~\ref{fct:sm_to_gal}, while the map $\Gal(T)\to p(\fC)/E$ is the orbit map $\sigma\Autf_L(\fC)\mapsto [\sigma(a)]_E$ (cf. Remark~\ref{rem:gal_action}).
		
		Recall that $E$ is type-definable [resp. Borel, analytic, $F_\sigma$, relatively definable] if and only if the induced relation $E^M$ on $S_a(M)$ is closed [resp. Borel, analytic, $F_\sigma$, clopen]. But $E^M$ is just the relation on $S_a(M)$ induced by the vertical arrow on the right, i.e.\ two types in $S_a(M)$ are $E^M$-equivalent if and only if they land in the same point in $p(\fC)/E$.
	
		Let us write $E', E''$ for the relations on $S_m(M)$ and $S_n(N)$ defined analogously (e.g.\ two types in $S_m(M)$ are $E'$-equivalent if they land in the same point in $p(\fC)/E$ via compositions of appropriate maps in the diagram). Note that they are both induced by the same left invariant equivalence relation on $\Gal(T)$ (because the map $\Gal(T)\to p(\fC)/E$ is left $\Gal(T)$-equivariant, as the orbit map of a left action).

		Abusing notation, let $E'$ be the relation on $[m]_{\equiv}$ such that $(E')^M$ is the $E'$ defined above. It is $\Aut(\fC)$-invariant by construction (e.g.\ because the equivalence relation on $\Gal(T)$ is left invariant), and it is clearly bounded by the size of $p(\fC)/E$. We will show that it satisfies the conclusion.
		
		The first part of the conclusion follows easily from Remark \ref{rem: continuous surjection is closed} and  the comments following Fact~\ref{fct:borel_section}. We will show now the remaining parts.
		
		Fact~\ref{fct:borel_section} gives us Borel sections $S_m(M)\to S_n(N)$ and $S_a(M)\to S_n(N)$ of the restriction maps, and we can assume without loss of generality that each section maps $\tp(m/M)$ or $\tp(a/M)$ (respectively) to $\tp(n/N)$ (possibly by changing the value of the section at one point). Those sections, composed with the appropriate restrictions from $S_n(N)$, yield Borel maps $r_1\colon S_m(M)\to S_a(M)$ and $r_2\colon S_a(M)\to S_m(M)$, which (by the last sentence) take $\tp(m/M)$ to $\tp(a/M)$ and vice versa. These maps are clearly Borel reductions between $E'$ and $E^M$ (passing via $E''$). Denote by $r_1',r_2'$ the induced maps between the class spaces (as in the statement of the lemma), where we freely identify various homeomorphic quotient spaces (e.g $p(\fC)/E$ and $S_a(M)/E^M$).
		
		Now, note that given an $\sigma\in \Aut(\fC)$, the restriction of $\tp(\sigma(n)/N)\in S_n(N)$ to $S_m(M)$ is $\tp(\sigma(m)/M)$ and likewise, the restriction to $S_a(M)$ is $\tp(\sigma(a)/M)$. It follows easily that for every $\sigma\in \Aut(\fC)$, we have $r_1'([\tp(\sigma(m)/M]_{E'})=[\tp(\sigma(a)/M)]_{E^M}$ and likewise, $r_2'([\tp(\sigma(a)/M)]_{E^M})=[\tp(\sigma(m)/M)]_{E'}$. In particular, $r_1'$ and $r_2'$ are bijections with $r_2'$ being the inverse of $r_1'$, and they are $\Gal(T)$-equivariant.
		
		Note that all the maps in the diagram are quotient maps, so in particular, the composed map $S_m(M)\to p(\fC)/E$ is a quotient map. It is easy to see that this map is the composition of the bijection $r_1'$ and the quotient map $S_m(M)\to [m]_{\equiv}/E'$, which implies that $r_1'$ is a homeomorphism, and hence, so is $r_2'$.

		Finally note that the conditions $r_1(\tp(m/M))=\tp(a/M)$ and $r_2(\tp(a/M))=\tp(m/M)$, together with $\Gal(T)$-equivariance of $r_1'$ and $r_2'$, imply that $r_1'$ and $r_2'$ are determined by $r_1'([\sigma(m)]_{E'})=[\sigma(a)]_{E}$ and $r_2'([\sigma(a)]_{E})=[\sigma(m)]_{E'}$, for all $\sigma \in \Aut(\fC)$.
	\end{proof}

	\subsection{Proof of the main theorem}\label{subsection: proof of the main theorem}
	The following is the formal and full statement of \hyperref[mainthm]{Main Theorem}. Recall that by Proposition \ref{prop:amb_exist} we know that for any countable theory $T$ and a countable subset $A$ of $\fC$ there is a countable ambitious model of $T$ containing $A$. We will be using the the notation set at the beginning of Section \ref{sec:main_thm}.
	\begin{thm}
		\label{thm:main}
		Let $T$ be an arbitrary countable theory.
		(Choose a countable ambitious model $M$ of $T$.) Then there is a compact Polish group $\hat G$ (namely, the quotient $u\cM/H(u\cM)\Core(D)$ considered in Corollary \ref{cor:polish_quotient}), and a topological group quotient mapping $\hat r\colon \hat G\to \Gal(T)$ (the one from Corollary \ref{cor:polish_quotient}), with the following property.
		
		Suppose $E$ is a strong type defined on $p(\fC)$ for some $p\in S(\emptyset)$ (in countably many variables). Fix any $a\models p$.
		
		Denote by $r_{[a]_E}$ the orbit map $\Gal(T)\to p(\fC)/E$ given by $\sigma \Autf_L(\fC)\mapsto[\sigma(a)]_E$ (i.e. the orbit map of the natural action of $\Gal(T)$ on $p(\fC)/E$ from Remark~\ref{rem:gal_action}).
		
		Then for $\hat r_{[a]_E}:=r_{[a]_E}\circ \hat r$ and $H=\ker \hat r_{[a]_E}:=\hat r_{[a]_E}^{-1}[[a]_E]$ we have that:
		\begin{enumerate}
			\item
			$H\leq \hat G$ and the fibers of $\hat r_{[a]_E}$ are the left cosets of $H$,
			\item
			$\hat r_{[a]_E}$ is a topological quotient map, and so $p(\fC)/E$ is homeomorphic to $\hat G/H$,
			\item
			$E$ is type-definable if and only if $H$ is closed,
			\item
			$E$ is relatively definable on $p(\fC) \times p(\fC)$ if and only if $H$ is clopen,
			\item
			if $E$ is Borel [resp. analytic, or $F_\sigma$], then $H$ is Borel [resp. analytic, or $F_\sigma$],
			\item
			$E_H\leq_B E$, where $E_H$ is the relation of lying in the same left coset of $H$.
		\end{enumerate}
		
		Moreover, if $T$ has NIP, or, more generally, if $T$ has a tame ambitious model $M$, then $E_H\sim_B E$.
	\end{thm}
	\begin{proof}
		Very roughly, we follow the blueprint of Theorem~\ref{thm:main_over_KP}. We use the notation from the beginning of Section \ref{sec:main_thm}.

		First of all, by Corollary \ref{cor:Polish_quotient_Core(D)}, we know that $\hat{G}:=u\cM/H(u\cM)\Core(D)$ is a compact Polish group (note that since $T$ and $M$ are countable, the ambit $S_m(M)$ is metrizable).

		Recall that $\hat r$ is the epimorphism from Corollary~\ref{cor:polish_quotient}. Since $r_{[a]_E}$ is an orbit mapping and $\hat r$ is an epimorphism, $\hat r_{[a]_E}$ is also an orbit mapping, which gives us the first point of the theorem.
		
		Corollary~\ref{cor:polish_quotient} also yields a topological quotient map $u\cM/H(u\cM)D \to \Gal(T)$ which composed with the topological quotient map $r_{[a]_E}$ (see Remark~\ref{rem:gal_action}) induces a topological quotient map $u\cM/H(u\cM)D \to p(\fC)/E$.
		
		Note that we have the following commutative diagram
		\begin{center}
			\begin{tikzcd}
				u\cM/H(u\cM)D=G'\ar[d, two heads] & \hat G=u\cM/H(u\cM)\Core(D)\ar[l, two heads]\ar[dl, two heads,"\hat r_{[a]_E}"]\ar[d,two heads,"\hat r"]\\
				 p(\fC)/E & \Gal(T)\ar[l,two heads, "r_{[a]_E}"],
			\end{tikzcd}
		\end{center}
		where $r_{[a]_E}$ and the top arrow are topological quotient mappings, as is $\hat r$ (by Corollary~\ref{cor:polish_quotient}), which by commutativity implies that so is $\hat r_{[a]_E}$.
		This together with (1) gives us the second point. 
		The third point follows from (2), Fact~\ref{fct:tdf_t2}, and Fact~\ref{fct:quotient_by_closed_subgroup}. The fourth point follows from (2) and easy observations that $E$ is relatively definable on $p(\fC) \times p(\fC)$ if and only if $p(\fC)/E$ is discrete, and that $H$ is clopen if and only if $\hat{G}/H$ is discrete.
		
		For brevity, write $G'$ for $u\cM/H(u\cM)D$ and $E_{G'}$ for the equivalence relation on $G'$ induced from equality on $p(\fC)/E$ by the left arrow. By Proposition \ref{prop:uM/HuMD_Polish} and Corollary \ref{cor:Polish_quotient_Core(D)}, $G'$ is a compact Polish space and $\hat{G}$ is a compact Polish group. Since $E_H$ is the preimage of $E_{G'}$ by the top arrow, we easily get that:
		\begin{itemize}
			\item
			Borelness [resp. analyticity, or being $F_\sigma$] of $H$, $E_H$ and $E_{G'}$ are all equivalent,
			\item
			$E_H\sim_B E_{G'}$.
		\end{itemize}
		Thus, for the fifth and sixth point, we only need to show that if $E$ is Borel [resp. analytic, or $F_\sigma$], then $E_{G'}$ is Borel [resp. analytic, or $F_\sigma$], and that $E_{G'}\leq_B E$.
		
			We may assume without loss of generality that $m\supseteq a$. Indeed, Lemma~\ref{lem:every_stype_on_m} yields a strong type $E'$ on $[m]_\equiv$ and a map $r_1'\colon [m]_\equiv/E' \to p(\fC)/E$ satisfying all the conclusions of that lemma, in particular, $r_1'([\sigma(m)]_{E'})= [\sigma(a)]_E$ for any $\sigma \in \Aut(\fC)$. Therefore, $r_{[a]_E}=r_1'\circ r_{[m]_{E'}}$, and so $\hat r_{[a]_E}=r_1'\circ \hat r_{[m]_{E'}}$. This, together with the conclusions of Lemma~\ref{lem:every_stype_on_m}, shows that we can even assume that $m=a$; however, we will only assume that $m\supseteq a$ (as this will be the case in the proof of Theorem \ref{thm:main_smaller}, in which we will refer to the current proof).

		In order to finish the proof of points (5) and (6), it is enough to see that we have the following diagram, satisfying the assumptions of Lemma~\ref{lem:main_1}.
		
		\begin{center}
			\begin{tikzcd}
				\overline{u\cM}/{\equiv}\ar[r, two heads]\ar[d] & G'=u\cM/H(u\cM)D\ar[d, two heads]\\
				S_a(M)\ar[r, two heads] & p(\fC)/E 
			\end{tikzcd}
		\end{center}
		
		\begin{itemize}
			\item
			The map $\overline{u\cM}/{\equiv}\to S_a(M)$ is the composition of the restriction map $S_m(M)\to S_a(M)$ and the map $\overline{u\cM}/{\equiv}\to S_m (M)$ induced by $R$ (if we identify $EL/{\equiv}$ with $S_m(M)$, the last map is just inclusion).
			\item
			The map $S_a(M)\to p(\fC)/E$ is the natural map from Fact~\ref{fct:stypes_from_types}.
			\item
			The map $\overline{u\cM}/{\equiv}\to u\cM/H(u\cM)D$ is induced by the continuous map $f\mapsto uf/H(u\cM)$ from Fact~\ref{fct:strange_cont} (see the proof of Proposition \ref{prop:uM/HuMD_Polish}).
			\item
			The arrow on the right is the same as the one on the left in the preceding diagram.
		\end{itemize}
		We know that $S_a(M)$, $\overline{u\cM}/{\equiv}$ and $G'$ are all compact Polish spaces. From the above items, the top and the left map in the last diagram are easily seen to be continuous. It is also easy to check that this diagram is commutative. Thus, the assumptions of Lemma \ref{lem:main_1} are satisfied, so the proof of (5) and (6) is complete.

		For the ``moreover" part, if $(\Aut(M),S_m(M))$ is tame, we have a similar diagram, this time satisfying the assumptions of Lemma~\ref{lem:main_2}, which yields the desired reduction $E \leq_B E_{G'} (\sim_B E_H)$.
		\begin{center}
			\begin{tikzcd}
				EL/{\equiv'}\ar[r, two heads,"Borel"]\ar[d, two heads] & G'=u\cM/H(u\cM)D\ar[d, two heads] \\
				S_a(M)\ar[r, two heads] & p(\fC)/E 
			\end{tikzcd}
		\end{center}
		Recall that $\equiv'$ is the closed equivalence relation on $EL$ defined by $f_1\equiv' f_2$ if $f_1(\tp(m/M))=f_2(\tp(m/M))$ and $f_1u(\tp(m/M))=f_2u(\tp(m/M))$. The left arrow is defined analogously to the previous diagram, and it is continuous. 
		The function $EL/{\equiv'}\to u\cM/H(u\cM)D$ is the map induced by $f\mapsto ufu$, $EL\to u\cM$. It is Borel by Corollary~\ref{cor:borel_map}; $EL/{\equiv'}$ is compact Polish by Proposition \ref{prop: quotients of EL are Polish}. Thus, the assumptions of Lemma \ref{lem:main_1} are indeed satisfied.
	\end{proof}

	Applying Theorem \ref{thm:main} to $a$ being an enumeration of $M$ and $E:= {\equiv_L}$, we get

	\begin{cor}\label{cor:main theorem for Gal}
		Let $T$ be a countable theory. As a topological group $\Gal(T)$ is the quotient of a compact Polish group $\hat{G}$ by an $F_\sigma$ normal subgroup $H$ (namely, $H:=\ker(\hat{r})$). Also, $\hat{G}/H \leq_B \Gal(T)$. Moreover, if $T$ has NIP, then $\hat{G}/H \sim_B \Gal(T)$.
	\end{cor}

	We have made essential use of the NIP assumption in proving the ``moreover" part of the theorem. However, we do not know counterexamples showing that it is necessary.
	\begin{ques}\label{ques:without_NIP}
		Does the ``moreover" part of Theorem~\ref{thm:main} hold without NIP? In particular, is it true that $E_{\ker \hat r}\sim_B \Gal(T)$ holds without NIP?
	\end{ques}

	It follows from Theorem~\ref{thm:main_over_KP} that the answer to Question \ref{ques:without_NIP} is positive for $G$-compact theories.
	
	\begin{rem}
		 In the NIP (or tame) case, Proposition \ref{prop:NIP gives metrizability} shows that $u\cM/H(u\cM)$ is already a Polish group, so we can use $u\cM/H(u\cM)$ in place of $u\cM/H(u\cM)\Core(D)$ in Theorem~\ref{thm:main} and in its proof (while still using $u\cM/H(u\cM)D$ in the proof).
		\xqed{\lozenge}
	\end{rem}

	The group $\hat G$ we have constructed depends (a priori) quite heavily on the choice of the model $M$ (even if $\hat G=u\cM/H(u\cM)$, as in the tame case). One could hope for a canonical compact Polish extension $\hat G$ of $\Gal(T)$ satisfying all (or most) of the properties stipulated in Theorem~\ref{thm:main}, at least under NIP --- if such $\hat G$ existed, it could be an interesting model-theoretic invariant (again, this is interesting primarily for non-G-compact theories, as otherwise, by Theorem~\ref{thm:main_over_KP}, we can take $\hat G:=\Gal_{KP}(T)$).
	
	\begin{ques}
		Suppose $T$ has NIP. Is there a canonical choice of $\hat G$ satisfying the properties listed in Theorem~\ref{thm:main}?
	\end{ques}

	\subsection{Corollaries}
	As promised in the introduction, we now derive Fact~\ref{fct:main_KPR} directly from Fact~\ref{fct:trich_polish}.
	\begin{cor}
		\label{cor:main_KPR}
		Fact~\ref{fct:main_KPR} holds.
	\end{cor}
	\begin{proof}
		We use the notation of Theorem~\ref{thm:main}.
		Since $\hat{G}$ is compact Polish and, by Theorem~\ref{thm:main}, $H$ is analytic (as we assume that $E$ is analytic), we can apply Fact~\ref{fct:trich_polish}, so exactly one of the following holds:
		\begin{enumerate}
			\item
			$H$ is open, and so $\hat{G}/H$ is finite (and smooth),
			\item
			$H$ is closed (so $\hat{G}/H$ is smooth) and $[\hat{G}:H]=2^{\aleph_0}$,
			\item
			$H$ is not closed, $\hat{G}/H$ is not smooth, and $[\hat{G}:H]=2^{\aleph_0}$.
		\end{enumerate}
		
		By Theorem~\ref{thm:main}, these yield that exactly one of the following holds:
		\begin{enumerate}
			\item
			$E$ is relatively definable, and so has finitely many classes (and is smooth),
			\item
			$E$ is type-definable (so smooth) and has $2^{\aleph_0}$ classes,
			\item
			$E$ is not type-definable and not smooth, and has $2^{\aleph_0}$ classes.\qedhere
		\end{enumerate}
	\end{proof}
	The following immediate corollary of Theorem~\ref{thm:main} yields an estimate on the Borel cardinality of the Lascar strong type (it is an open problem whether it can be anything other than smooth, $\EZ$ or $\ell^\infty$, see \cite{KPS13}).
	\begin{cor}
		\label{cor:bc_restriction}
		Suppose $T$ is a countable NIP theory. Then there is a compact Polish group $\hat{G}$ with the following properties.
		\begin{itemize}
			\item
			If $E$ is a bounded invariant equivalence relation on the set $p(\fC)$ for some $p\in S(\emptyset)$ in countably many variables, its Borel cardinality is the same as the Borel cardinality $\hat{G}/H$ for $H\leq \hat{G}$ which is closed [resp. Borel, or analytic, or $F_\sigma$] 
			whenever $E$ is. (In particular, if $E={\equiv_L}$, then $H$ is $F_\sigma$.)
			\item 
			In particular, the Borel cardinality of $\Gal(T)$ equals the Borel cardinality of $\hat{G}/H$, where $H$ is some $F_\sigma$ normal subgroup of $\hat{G}$. \xqed{\lozenge}
		\end{itemize}
	\end{cor}
	
	\section{Variants of the main theorem}
	In this section, we present some variants of the main theorem which can be obtained by similar methods.

	\subsection{Smaller domains}\label{subsection: smaller domains}
	Given a $\equiv_L$-invariant set $Y\subseteq \fC$, denote by $\Gal(T/\{Y\})$ the quotient of $\Aut(\fC/\{Y\})$ (the group of automorphisms fixing $Y$ setwise) by $\Autf_L(\fC)$ (note that $\Autf_L(\fC)$ is always a subgroup of $\Aut(\fC/\{Y\})$, by the Lascar invariance assumption). Then $\Gal(T/\{Y\})$ is clearly a subgroup of $\Gal(T)$, and thus a topological group (with the subspace topology). Furthermore, notice that if $Y$ is type-definable, it is already type-definable over any elementary submodel of $\fC$ (by $\equiv_L$-invariance), and moreover, $\Gal(T/\{Y\})$ is a closed subgroup of $\Gal(T)$ (see \cite[Lemma 1.17]{KPR15}).

	This allows us to ``extend" the main theorem to the following form. (Strictly speaking, the next theorem is not a generalization of Theorem~\ref{thm:main}, because of the order of quantifiers: the compact Polish group $\hat{G}$ in Theorem~\ref{thm:main}  was good for all types $p$, whereas in the next theorem it depends on $p$ (and $Y$).)
	\begin{thm}
		\label{thm:main_smaller}
		Let $T$ be an arbitrary countable theory. 
		Suppose $Y$ is an $\equiv_L$-invariant, countably supported and type-definable (with parameters) subset of some $p(\fC)$. Suppose in addition that $\Aut(\fC/\{Y\})$ acts transitively on $Y$ (this is true e.g.\ when $Y=p(\fC)$, or $Y$ is a single KP or Shelah strong type), and fix some $a\in Y$.
		
		Then there is a compact Polish group $\hat G_Y$ and a topological group quotient mapping $\hat r_Y\colon \hat G_Y\to \Gal(T/\{Y\})$, with the following property.
		
		Suppose $E$ is a strong type on $p(\fC)$ such that $Y$ is $E$-invariant.
		
		Denote by $r_{[a]_E,Y}$ the orbit map $\Gal(T/\{Y\})\to Y/E$ given by $\sigma\Autf_L(\fC)\mapsto[\sigma(a)]_E$ (the restriction of the map from Theorem~\ref{thm:main}; note that it is onto by the assumption about transitivity of the action of $\Aut(\fC/\{Y\})$ on $Y$).
		
		Then for $\hat r_{[a]_E,Y}:=r_{[a]_E,Y}\circ \hat r_Y$ and $H_Y=\ker \hat r_{[a]_E,Y}:=\hat r_{[a]_E,Y}^{-1}[[a]_E]$ we have that:
		\begin{enumerate}
			\item
			$H_Y\leq \hat G_Y$ and the fibers of $\hat r_{[a]_E,Y}$ are the left cosets of $H_Y$,
			\item 
			$\hat r_{[a]_E,Y}$ is a topological quotient mapping, and so $Y/E$ is homeomorphic to $\hat G_Y/H_Y$,
			\item
			$E \restr_Y$ is type-definable if and only if $H_Y$ is closed,
			\item
			$E \restr_Y$ is relatively definable on $Y \times Y$ if and only if $H_Y$ is clopen,
			\item
			if $E \restr_Y$ is Borel [resp. analytic, or $F_\sigma$], then so is $H_Y$,
			\item
			$E_{H_Y}\leq_B E\restr_Y$, where $E_{H_Y}$ is the relation of lying in the same left coset of ${H_Y}$.
		\end{enumerate}
		Moreover, if $T$ has NIP, then $E_{H_Y}\sim_B E\restr_Y$.
	\end{thm}
	\begin{proof}
		Fix $Y,a$ satisfying the assumptions. Let $M$ be a countable ambitious model containing $a$, enumerated as $m$. 
		
		Take $\hat G, \hat r$ as in the proof of Theorem~\ref{thm:main}. Let $\hat G_Y:=\hat r^{-1}[\Gal(T/\{Y\})]$ and $\hat r_Y:=\hat r\restr_{\hat G_Y}$. We will show that they satisfy the conclusion.
		
		First, notice that $\hat{G}_Y$ is a compact Polish group and $\hat r_Y$ is a topological group quotient map $\hat G_Y\to \Gal(T/\{Y\})$ (as the restriction of the quotient map $\hat r$ to the preimage of a closed subgroup).
		
		Consider any strong type $E$ such that $Y$ is $E$-invariant.

		Then $H_Y=\ker \hat r_{[a]_E,Y}=\ker \hat r_{[a]_E}\cap \hat G_Y$, so the first point follows from point (1) in Theorem~\ref{thm:main}. 
		
		Let $S_{Y,m}(M):=\{\tp(\sigma(m)/M)\mid \sigma\in \Aut(\fC/\{Y\}) \}$, and recall that $Y_M$ is the space of types over $M$ of elements of $Y$. It is easy to see that $S_{Y,m}(M)$ is a closed subspace of $S_m(M)$, and we have a commutative diagram
		\begin{center}
			\begin{tikzcd}
				S_m(M)&[-3em]\supseteq&[-3em]S_{Y,m}(M)\ar[r, two heads]\ar[d, two heads] & \Gal(T/\{Y\})\ar[d, two heads,"r_{[a]_E,Y}"]&[-3em]\leq&[-3em] \Gal(T) \\
				S_a(M)&\supseteq&Y_M\ar[r, two heads] & Y/E&\subseteq& p(\fC)/E,
			\end{tikzcd}
		\end{center}
		where the left arrow is the restriction to fewer variables (which is well-defined, because we have assumed that $a$ is a subtuple of $m$), while the horizontal arrows are simply the restrictions of the natural map $S_a(M)\to p(\fC)/E$ and the map $S_m(M)\to \Gal(T)$ from Fact~\ref{fct:sm_to_gal}.
		The left map is onto, because $\Aut(\fC/\{Y\})$ acts transitively on $Y$.
		It is clear that the horizontal arrows are surjective as well.
		The top arrow is a quotient map as the restriction of a quotient map to the preimage of a closed set. The bottom one is a quotient map by Definition \ref{dfn:logic topology}, and the left one is quotient as a continuous surjection between compact Hausdorff spaces. It follows that $r_{[a]_E,Y}$ is also a quotient map, and since $\hat r_Y$ is a quotient map, this implies that so is $\hat r_{[a]_E,Y}=r_{[a]_E,Y}\circ \hat r_{Y}$, which gives us the second point of the theorem.

		The third and fourth point follow from (2) as in the proof of Theorem~\ref{thm:main}.	

		Define $G'_Y$ as the preimage of $\Gal(T/\{Y\})$ via the map $u\cM/H(u\cM)D \to \Gal(T)$ from Corollary~\ref{cor:polish_quotient}. Then the proof of points (5) and (6) in Theorem~\ref{thm:main} goes through with $G'$, $\hat G$, $p(\fC)/E$, $\Gal(T)$, $S_a(M)$ replaced by $G'_Y$, $\hat{G}_Y$, $Y/E$, $\Gal(T/\{Y\}), Y_M$, respectively, and with $\overline{u\cM}/{\equiv}$ replaced by the preimage of $G'_Y$ via the top arrow in the second diagram in the proof of Theorem~\ref{thm:main}.
		
		To show the ``moreover" part, apply the proof of the ``moreover" part of Theorem~\ref{thm:main}, with the same replacements as above, and with $EL/{\equiv}'$ replaced by $EL^Y/{\equiv'}$, where $EL^Y$ is the set of $f\in EL$ such that 
		$f(\tp(m/M))\in S_{Y,m}(M)$.
		In order to do that, make the following observations:
		\begin{itemize}
			\item 
			as $S_{Y,m}(M)$ is closed in $S_m(M)$, $EL^Y$ is closed in $EL$, so $EL^Y/{\equiv'}$ is compact,
			\item
			the Borel map $EL/{\equiv'}\to u\cM/H(u\cM)D$ ($[f]_{\equiv'}\mapsto ufuH(u\cM)D$), restricted to $EL^Y/{\equiv'}$, is onto $G'_Y$,
			\item
			the arrow $EL^Y/{\equiv'}\to Y_M$ corresponding to the left map in the last diagram of the proof of Theorem~\ref{thm:main} is onto --- this follows from surjectivity of $S_{Y,m}(M)\to Y_M$ in the diagram above.\qedhere
		\end{itemize}
	\end{proof}
	
	As a conclusion, we get an obvious extension of Corollaries \ref{cor:main theorem for Gal} and \ref{cor:bc_restriction}.
	
	\begin{cor}
		Let $T$ be a countable theory. Then $\Gal_0(T)$ is the quotient of a compact Polish group $G$ by an $F_\sigma$, dense, normal subgroup $H$. Also, $G/H \leq_B \Gal_0(T)$. Moreover, if $T$ has NIP, then $G/H \sim_B \Gal_0(T)$.
	\end{cor}

	\begin{proof}
		Apply Theorem~\ref{thm:main_smaller} to $a$ enumerating a countable ambitious model $M$, $Y:= [a]_{\equiv_{KP}}$, and $E:={\equiv_L}$ on $p(\fC)$. Then $G:= \hat G_Y$ and $H:=H_Y$ work.
	\end{proof}

	\begin{cor}
		\label{cor:bc_restriction_smaller}
		Suppose $T$ is a countable NIP theory, while $Y$ is as in Theorem~\ref{thm:main_smaller}. Then there is a compact Polish group $G$ such that if $E$ is a bounded invariant equivalence relation as in Theorem~\ref{thm:main_smaller}, the Borel cardinality of $Y/E$ is the same as that of $G/H$ for some $H\leq G$ which is closed [resp. Borel, or analytic, or $F_\sigma$] 
		whenever $E$ is. (In particular, if $E={\equiv_L}$, $H$ is $F_\sigma$.)\xqed{\lozenge}
	\end{cor}

	\subsection{Type-definable groups}\label{subsection: type-definable groups}

	Recall that for a $\emptyset$-type-definable group $G$, $G^{00}_\emptyset$ [and $G^{000}_\emptyset$] denote the smallest bounded index, $\emptyset$-type-definable [resp. invariant] subgroup of $G$. Equipped with the logic topology, $G/G^{000}_\emptyset$ is a compact (not necessarily Hausdorff) group, while $G/G^{00}_\emptyset$ is a compact Hausdorff group; in fact, $G^{00}_\emptyset/G^{000}_\emptyset$ is the closure of the identity in $G/G^{000}_\emptyset$. For details on these issues, see e.g. \cite{Gis11, GiNe08}. 
	
	Assume that the language is countable. We say that an invariant subgroup $K$ of $G$ is {\em Borel} [resp. {\em analytic}, or $F_\sigma$] if $K_\emptyset$ is such in the type space $S_G(\emptyset)$. Let $E_K$ be the invariant equivalence relation of lying in the same left coset of $K$. Then $K$ is Borel [resp. analytic, or $F_\sigma$] if and only if $E_K$ is (in the sense of Definition \ref{dfn:Borel strong types}). By the {\em Borel cardinality} of $G/K$ we mean the Borel cardinality (in the sense of Definition \ref{dfn: Borel cardinality for strong types}) of the relation $E_K$ .
	
	Having general results on the complexity of $\Gal(T)$ or strong types, one can usually get as an easy corollary analogous results for quotients of a definable group by invariant subgroups of bounded index. This is achieved by expanding the original structure by the affine copy of $G$ as a new sort (see \cite[Section 1.5]{KPR15}). However, here we want to establish results for a type-definable group $G$ in which case the aforementioned trick does not work. So one has to prove counterparts of the results of Section 7 for type-definable groups. In any case, we get the following counterpart of our main theorem.

	\begin{thm}
		\label{thm:main_group}
		Suppose $G$ is a $\emptyset$-type-definable group (in countably many variables, in a countable theory $T$).
		
		Then there is a compact Polish group $\hat G$ and a topological group quotient mapping $\hat r\colon \hat G\to G/G^{000}_\emptyset$ such that for any $K\leq G$ invariant of bounded index, the map $\hat r_K\colon \hat G\to G/K$ (which is the composition of $\hat r$ and the natural quotient map $r_K\colon G/G^{000}_\emptyset\to G/K$) and $H:=\hat r^{-1}[K/G^{000}_\emptyset]$, we have that:
		\begin{enumerate}
			\item
			$H\leq \hat G$ and the fibers of $\hat r_K$ are the left cosets of $H$,
			\item
			$\hat r_K$ is a topological quotient mapping, so $G/K$ is homeomorphic to $\hat{G}/H$ (where $G/K$ is equipped with the logic topology, and $\hat G/H$ with the quotient topology),
			\item
			$K$ is type-definable if and only if $H$ is closed,
			\item
			$K$ is relatively definable in $G$ if and only if $H$ is clopen,
			\item
			if $K$ is Borel [resp. analytic, or $F_\sigma$], so is $H$,
			\item
			$E_H\leq_B E_K$, where $E_H$ is the relation of lying in the same left coset of $H$ in $\hat G$, and $E_K$ is the relation of lying in the same left coset of $K$ in $G$.
		\end{enumerate}
		
		Moreover, if $T$ has NIP (or, more generally, if there is a countable model $M$ such that $(G(M),S_G(M))$ is tame and $G(M)\cdot \tp(e/M)$ is dense in $S_G(M)$), then $E_H\sim_B E_K$.
	\end{thm}
	\begin{proof}
		The proof is analogous to the proof of Theorem~\ref{thm:main}. Namely, we take a countable model $M$ such that $G(M)\cdot \tp(e/M)$ is dense in $S_G(M)$ --- the construction of such model is similar to the construction of an ambitious model in Proposition~\ref{prop:amb_exist} --- and consider the dynamical system $(G(M),S_G(M))$. Since $(G(M),S_G(M),\tp(e/M))$ is an ambit, all the considerations in Section~\ref{sec:top_dyn_to_Polish} apply.
		
		Then we follow the blueprint of Section~\ref{sec:main_thm}, making modifications similar to the original ideas of \cite{KP17}.
		\begin{itemize}
			\item
			$G=G(\fC)$ takes place of $\Aut(\fC)$, while $G/G^{000}_\emptyset$ takes place of $\Gal(T)$.
			\item
			We have a counterpart of Lemma~\ref{lem:pseudocompactness}, with a similar proof.
			\item
			To define a counterpart of the map $r$, we use the fact that if $g_1,g_2\in G$ have the same type over $M$, then $g_1^{-1}g_2\in G^{000}_\emptyset$.
			\item
			In the proof of Proposition~\ref{prop:H(uM)_in_ker}, instead of using Lascar distance directly, we use powers of the set $\{g_1^{-1}g_2\mid g_1,g_2\in G, d_L(g_1,g_2)\leq 1 \}$, where $d_L$ is the Lascar distance (as in the proof of \cite[Theorem 0.1]{KP17}).
			\item
			In proving the counterpart of Proposition~\ref{prop:id_clsd}, once again we use the fact that if $g_1,g_2\in G$ have the same type over $M$, then $g_1^{-1}g_2\in G^{000}_\emptyset$.
			\item
			In the proof of Proposition~\ref{lem:r_restr_to_top_quot}, we use the fact that $G^{00}_\emptyset/G^{000}_\emptyset$ is the closure of the identity in $G/G^{000}_\emptyset$.
			\item
			Other parts of the proof are almost the same.\qedhere
		\end{itemize}
	\end{proof}
	
	Applying this theorem to $K:=G^{000}_\emptyset$, we get
	
	\begin{cor}\label{cor:quotient by G000}
		Suppose $G$ is a $\emptyset$-type-definable group (in countably many variables, in a countable theory $T$). Then $G/G^{000}_\emptyset$ is homeomorphic to the quotient of a compact Polish group $\hat{G}$ by an $F_\sigma$ normal subgroup $H$. Also $\hat{G}/H \leq_B G/G^{000}_\emptyset$. Moreover, if $T$ has NIP, then $\hat{G}/H \sim_B G/G^{000}_\emptyset$.\xqed{\lozenge}
	\end{cor}
	
	The following corollary is a generalization of \cite[Corollary 4.7]{KPR15}; in particular, the group $G$ need not be definable (or a subgroup of a definable group).

	\begin{cor}
		\label{cor:main_group}
		Suppose $G$ is a $\emptyset$-type-definable (countably supported) group in a countable theory, while $K\leq G$ is Borel (or, more generally, analytic), invariant of bounded index. Then exactly one of the following conditions holds:
		\begin{itemize}
			\item
			$K$ is relatively definable (so $G/K$ is smooth) and $[G:K]<\infty$,
			\item
			$K$ is type-definable (so $G/K$ is smooth) and $[G:K]=2^{\aleph_0}$,
			\item
			$K$ is not type-definable, $G/K$ is not smooth, and $[G:K]=2^{\aleph_0}$. 
		\end{itemize}
	\end{cor}
	\begin{proof}
		Analogous to the proof of Corollary~\ref{cor:main_KPR}. 
	\end{proof}
	\begin{rem}
		Clearly, Theorem~\ref{thm:main_group} and Corollaries \ref{cor:quotient by G000} and \ref{cor:main_group} remain true if we name countably many constants, thus $\emptyset$-type-definable and ($\emptyset$-)invariant can be replaced with ``$A$-type-definable" and ``$A$-invariant" for an arbitrarily chosen countable set $A$ of parameters.
	\end{rem}

\appendix
	\section{Examples}
	We analyze examples of non-G-compact theories $T$ from \cite{CLPZ01} and \cite{KPS13} and see how Theorem~\ref{thm:main} can be applied to them. Namely, we describe the compact group $\hat G$ (which will turn out to be the Ellis group) and the kernel of $\hat r\colon \hat G\to \Gal(T)$ in those cases. This allows us to compute the Borel cardinality of $\Gal(T)$ in these examples (which was also computed \cite[Remark 5.3]{KPS13}, but in a different way and without giving details).

	In this section, unless otherwise stated, $M_n$ denotes the countable structure $(M_n,R_n,C_n)$, where $n>1$ is a fixed natural number, the underlying set is ${\bbQ}/{\bbZ}$, $R_n$ is the unary function $x\mapsto x+1/n$, and $C_n$ is the ternary predicate for the natural (dense, strict) circular order. Let a tuple $m_n$ enumerate $M_n$. It is easy to show (see \cite[Proposition 4.2]{CLPZ01}) that $\Th(M_n)$ has quantifier elimination and the real circle $S^1_n=\bbR/\bbZ$ equipped with the rotation by the angle $2\pi/n$ and the circular order is an elementary extension of $M_n$. As usual, $\fC \succ S^1_n$ is a monster model.
	
	Given any $c' \in \fC$, by $\st(c')$ we denote the standard part of $c'$ computed in the circle $S^1=\bbR/\bbZ$.
	As $\st(c')$ depends only on $\tp(c'/M_n)$, this extends to a standard part mapping on the space of 1-types $S_1(M_n)$.
	
	\begin{prop}
		\label{prop:group_onetypes}
		The Ellis group of $(\Aut(M_n),S_1(M_n))$ is isomorphic to ${\bbZ}/n{\bbZ}$.
	\end{prop}

	\begin{proof}
		In this proof, by short interval we mean an interval of length less than $1/n$. We also identify $\Aut(M_n)$ with its image in the Ellis semigroup.
		
		Note that $R_n$ is a $\emptyset$-definable automorphism of $M_n$, and as such, it is in the center of $\Aut(M_n)$, and so it is also central in the Ellis semigroup.
		
		From quantifier elimination, it follows easily that $M_n$ is $\omega$-categorical, and $\Aut(M_n)$ acts transitively on the set of short open intervals in $M_n$.
		
		Denote by $J$ the set of $p\in S_1(M_n)$ with $\st(p)\in [0,1/n) + \bbZ \subseteq \bbR/\bbZ$.
		
		\begin{clm}
			For any non-isolated type $p$, there is a unique $f_{p}\in EL:=E(\Aut(M_n),S_1(M_n))$ such that for all $q\in J$ we have $f_p(q)=p$.
		\end{clm}
		\begin{clmproof}
			Enumerate $M_n$ as $(a_k)_{k\in \bbN}$.
			
			Since $p$ is non-isolated, for each $k\in \bbN$ there is a short open interval $I_k$ such that $p$ is concentrated on $I_k$ and $a_0,\ldots,a_k\notin I_k$. By quantifier elimination, it is easy to see that $p$ is the only type in $S_1(M_n)$ concentrated on all $I_k$'s.
			
			Now, let $J_k:=(\frac{-1}{2kn},\frac{1}{n}-\frac{1}{kn})$. Notice that if $q\in J$, then $q$ is concentrated on all but finitely many $J_k$'s.
			
			Since each $I_k$ and $J_k$ is a short open interval, we can find for each $k$ some $\sigma_k\in \Aut(M_n)$ such that $\sigma_k[J_k]=I_k$. It follows that for any $q\in J$ we have $\lim_k\sigma_k(q)=p$. Thus, if we take any $f_p\in EL$ which is a limit point of $(\sigma_k)_k$, we will have $f_p(q)=p$ for all $q \in J$.
			
			To see that $f_p$ is unique, note that for each integer $j$ and $q\in R_n^j[J]$, $f_p(q)\in f_p[R_n^j[J]]=f_pR_n^j[J]=R_n^j f_p[J]=\{R_n^j(p)\}$. Since $J\cup R_n[J]\cup\ldots\cup R_n^{n-1}[J]=S_1(M_n)$, uniqueness follows.
		\end{clmproof}
	
		Take any non-isolated $p_0\in J$, and let $u=f_{p_0}$ (as in the claim). By uniqueness in the claim, $u$ is an idempotent. Denote by $\mathcal O$ the $R_n$-orbit of $p_0$.
		
		Note that every $f \in ELu$ is constant on $J$. As in the above proof of uniqueness, since $u$ and $uf$ commute with $R_n$, we easily see that the image of $uf$ equals $\mathcal O$.
	
		Now, we show that $\cM:=ELu$ is a minimal left ideal. Consider any $f \in \cM$. 
		By the last paragraph, $uf(p_0)=R_n^j(p_0)$ for some $j$. Then $R_n^{-j}uf(p_0)=p_0$ and $R_n^{-j}uf$ is constant on $J$, so by uniqueness in the claim, $R_n^{-j}uf=u$. It follows that $ELf=ELu=\cM$, so $\cM$ is a minimal left ideal.
		
		Finally, $u\cM$ acts faithfully on $\mathcal O$ (since each $f\in u\cM$ is constant on $J$, $R_n[J],\ldots$, it is determined by its values on $\mathcal O$). As elements of $u\cM$ commute with $R_n$, we see that they act on $\mathcal O$ as powers of $R_n$. Since $R_nu = uR_nu \in u\cM$ acts as $R_n$, we get that $u\cM \cong \bbZ/n\bbZ$.
	\end{proof}
	
	\begin{lem}
		\label{lem:ellis_group_onesort}
		Suppose $n>1$.
		
		The restriction $S_{m_n}(M_n)\to S_1(M_n)$ to the first variable induces an isomorphism of Ellis semigroups $E(\Aut(M_n),S_{m_n}(M_n))\cong E(\Aut(M_n),S_1(M_n))$
		
		In particular, the Ellis group of $(\Aut(M_n),S_{m_n}(M_n))$ is isomorphic to $\bbZ/n\bbZ$.
	\end{lem}
	\begin{proof}
		We have the following ``orthogonality" property.
		
		\begin{clm*}
			Let $p,q\in S_{m_n}(M_n)$ satisfy the condition that for each single variable $x$, $p\restr_x=q\restr_x$. Then $p=q$.
		\end{clm*}
		\begin{clmproof}
			For $c_1',c_2'\in \fC$, write $c_1'<c_2'$ for $C_n(c_1',c_2',R_n(c_1'))$. Note that for each $r\in S^1$, this is a linear ordering on the set of all $c'$ with $\st(c')=r$. Furthermore, for any $c_1',c_2',c_3'$ we have that $C_n(c_1',c_2',c_3')$ holds if and only if one of the following holds:
			\begin{itemize}
				\item
				$\st(c_1'),\st(c_2'),\st(c_3')$ are all distinct and they are in the standard circular order on $S^1$,
				\item 
				$\st(c_1')=\st(c_2')\neq \st(c_3')$ and $c_1'<c_2'$,
				\item 
				$\st(c_1')\neq\st(c_2')=\st(c_3')$ and $c_2'<c_3'$,
				\item 
				$\st(c_2')\neq\st(c_1')=\st(c_3')$ and $c_1'>c_3'$,
				\item 
				$\st(c_1')=\st(c_2')=\st(c_3')$ and ($c_1'<c_2'<c_3'$ or $c_3'<c_1'<c_2'$ or $c_2'<c_3'<c_1'$).
			\end{itemize}
			
			We need to show that for each $m'=(m'_k)_{k\in \bbN}$ satisfying $\tp(m_n/\emptyset)$, we have $\tp(m_n/\emptyset)\cup\bigcup_k \tp(m'_k/M_n)\vdash \tp(m'/M_n)$. By quantifier elimination, it is enough to show that the type on the left implies each atomic formula (or negation) in $\tp(m'/M_n)$. The only nontrivial cases are of the form $C_n(R_n^i(x_1),R_n^j(x_2),c)$, $C_n(R_n^i(x_1),c,R_n^j(x_2))$, $C_n(c,R_n^i(x_1),R_n^j(x_2))$ (or negations), where $i,j \in \{0,\dots,n-1\}$ and $c \in M_n$. But that follows immediately from the preceding paragraph (and the fact that the standard part is determined by the type over $M_n$).
		\end{clmproof}
		
		It follows from quantifier elimination that there is a unique 1-type over $\emptyset$, so the restriction to the first variable $S_{m_n}(M_n)\to S_1(M_n)$ is surjective, and (since it is obviously equivariant) it gives us a surjective homomorphism $E(\Aut(M_n),S_{m_n}(M_n))\to E(\Aut(M_n),S_1(M_n))$. We need to show that it is injective.

		Suppose $f_1,f_2\in E(\Aut(M_n),S_{m_n}(M_n))$ are distinct, so there is some $p\in S_{m_n}(M_n)$ such that $f_1(p)\neq f_2(p)$. But then, by the claim, there is a variable $x_k$ such that $f_1(p)\restr x_k\neq f_2(p)\restr x_k$. Choose $m'=(m'_k)_{k\in \bbN}\models p$; then $m'$ enumerates a countable $M'\preceq \fC$. By $\omega$-categoricity and the fact that there is a unique 1-type over $\emptyset$, there is $\sigma\in \Aut(M')$ such that $\sigma(m'_1)=m'_k$. Now, if we put $p':=\tp(\sigma(m')/M_n)$, we have that $p'\restr_{x_1}=p\restr_{x_k}$. From that, we obtain $f_1(p')\restr_{x_1}=f_1(p)\restr_{x_k}\neq f_2(p)\restr_{x_k}=f_2(p')\restr_{x_1}$. It follows that the epimorphism $E(\Aut(M_n),S_{m_n}(M_n))\to E(\Aut(M_n),S_1(M_n))$ induced by the restriction to the first variable is injective, so we are done.
	\end{proof}

	\begin{prop}\label{prop:product of Ellis groups}
		Suppose we have a multi-sorted structure $M=(M_n)_n$, where the sorts $M_n$ are arbitrary, without any functions or relations between them. Enumerate each $M_n$ by $m_n$ and put $m=(m_n)_n$. Then $E(\Aut(M),S_m(M))\cong \prod_n E(\Aut(M_n),S_{m_n}(M))$, and similarly, the minimal left ideals and the Ellis groups (equipped with the $\tau$-topology) are the products of minimal left ideals and Ellis groups, respectively.
	\end{prop}
	\begin{proof}
		There is an obvious isomorphism $\Aut(M)\cong \prod_n\Aut(M_n)$ and homeomorphism $S_{m}(M)\approx \prod_nS_{m_n}(M_n)$, which together yield an isomorphism $(\Aut(M), S_m(M)) \cong (\prod_n\Aut(M_n), \prod_nS_{m_n}(M_n))$. This gives us an isomorphism $E(\Aut(M),S_m(M))\cong E(\prod_n \Aut(M_n),\prod_n S_{m_n}(M))$, and the last semigroup is easily seen to be isomorphic to $\prod_n E(\Aut(M_n),S_{m_n}(M))$.
		
		The corresponding statements about minimal left ideals and Ellis groups are straightforward consequences, except the fact that the induced isomorphism for Ellis groups is topological, which requires some work and is left as an exercise.
	\end{proof}
	
	\begin{ex}
		Consider the theory $T$ of the multi-sorted structure $M=(M_n)_{n>1}$, where each $M_n$ is the countable model as described at the beginning of this section. Then, if we enumerate $M$ as $m$, then $M$ is ambitious (because it is $\omega$-categorical). By Lemma \ref{lem:ellis_group_onesort} and Proposition \ref{prop:product of Ellis groups}, the Ellis group $u\cM$ of $(\Aut(M),S_m(M))$ is $\prod_n \bbZ/n\bbZ$ with the product topology. In particular, it is a Hausdorff (compact and Polish) group, so $H(u\cM)$ is trivial. 

		Moreover, the group $D=[u]_\equiv\cap u\cM$ is trivial. Indeed, if $f\in u\cM$ is nontrivial, then, for some $n$, $f\restr_{S_{m_n}(M_n)} \ne u\restr_{S_{m_n}(M_n)}$. Therefore, by Lemma \ref{lem:ellis_group_onesort}, the restriction $f\restr_{S_1(M_n)}$ to the first coordinate of $m_n$ is distinct from $u\restr_{S_1(M_n)}$. On the other hand, the argument after the claim in the proof of Proposition \ref{prop:group_onetypes} easily shows that $f\restr_{S_1(M_n)}=R_n^ju\restr_{S_1(M_n)}$ for some $j \in \{0,\dots,n-1\}$. Hence, $j \ne 0$. Thus, $f(\tp(m/M))\restr_x=R_n^ju \restr_{S_1(M_n)}(\tp(m_n^0/M_n))\ne u \restr_{S_1(M_n)}(\tp(m_n^0/M_n))$, where $m_n^0$ is the first coordinate of $m_n$ and $x$ is the corresponding variable. Hence, $f(\tp(m/M)) \ne u(\tp(m/M))$, i.e. $ f \notin [u]_\equiv$.

		 We have proved that $u\cM/H(u\cM)D=u\cM/H(u\cM)=u\cM \cong \prod_n \bbZ/n\bbZ$, so the group $\hat G$ from Theorem~\ref{thm:main} is $u\cM$, which we identify with $\prod_n \bbZ/n\bbZ$. Now, any $g\in \hat G$ can be uniquely represented as a sequence $(g_n)_{n\in \bbN}$, where $g_n$ is an integer in the interval $(-n/2,n/2]$.

		We claim that $g\in \ker \hat r$ if and only if the $g_n$'s are absolutely bounded.

		By \cite[Corollary~4.3]{CLPZ01}, for any $a\in M_n(\fC)$ and integer $k \in (-n/2,n/2]$ we have $d_L(a,R_n^k(a)) \geq k$, which easily implies (having in mind the precise identification of $u\cM$ with $\prod_n \bbZ/n\bbZ$) that unbounded sequences are not in the kernel.
	
		On the other hand, to show that absolutely bounded sequences are in $\ker \hat r$, it is enough to show this for sequences bounded by $1$. But for an element $f \in u\cM$ corresponding to such a sequence, the argument after the claim in the proof of Proposition \ref{prop:group_onetypes} and Lemma \ref{lem:ellis_group_onesort} easily yield that for every $n$, $f\restr_{S_{m_n}(M_n)} =R_n^{\epsilon_n}u \restr_{S_{m_n}(M_n)}=u \restr_{S_{m_n}(M_n)}R_n^{\epsilon_n}$ for some $\epsilon_n \in \{-1,0,1\}$. By \cite[Lemma 3.7]{CLPZ01}, it is enough to show that $d_L(m_n,R_n(m_n))$ is bounded (when $n$ varies). 
		By $\omega$-categoricity, we can replace $m_n$ by an enumeration $m_n'$ of any other countable model $M_n'$. So let $m'_n$ be an enumeration of $(\bbQ\cap ([0,1/3n)+\bbZ/n ))/\bbZ\subseteq \bbQ/\bbZ$. Furthermore, put $m''_n:=m'_n+1/3n$ and $m'''_n:=m'_n+2/3n$, and write $M'_n,M''_n,M'''_n$ for the respective models they enumerate. Then $\tp(m'_n/M'''_n)=\tp(m''_n/M'''_n)$, $\tp(m''_n/M'_n)=\tp(m'''_n/M'_n)$, $\tp(m'''_n/M''_n)=\tp(R_n(m'_n)/M''_n)$, so $d_L(m_n',R_n(m_n')) \leq 6$.

		Note that $T$ has NIP (e.g.\ because it is interpretable in an o-minimal theory), so the full Theorem~\ref{thm:main} applies, and the Galois group $\Gal(T)$ is the quotient of $\prod_n \bbZ/n\bbZ$ by the subgroup of bounded sequences. As a topological group, this is exactly the description given by \cite[Theorem~28]{Zie02}; note that the topology is trivial. In terms of Borel cardinality, we obtain $\ell^\infty$ 
		(see the paragraph following the proof of Lemma 3.10 in \cite{KPS13}).
	\end{ex}
	
	\begin{ex}
		Consider the theory $T$ of the multi-sorted structure $M=(M_n,h_{nn'})_{n,n'}$, where $M_n$ are as before, $n$ runs over the integers greater than $1$, while $n'$ ranges over the integers greater than $1$ and dividing $n$; for each pair $n' \mid n$, $h_{nn'}\colon M_n\to M_{n'}$ is the multiplication by $n/n'$. Enumerate each $M_n$ by $m_n$ and $M$ by $(m_n)_n$.

		Each $h_{nn'}$ induces a natural epimorphism $\Aut(M_n)\to \Aut(M_{n'})$, and using that, it is not hard to see that $\Aut(M)\cong \varprojlim_n \Aut(M_n)$. Similarly, $h_{nn'}$ induces a continuous, $\Aut(M_n)$-equivariant surjection $S_{m_n}(M_n)\to S_{m_{n'}}(M_{n'})$ (where $\Aut(M_n)$ acts on $S_{m_{n'}}(M_{n'})$ via $\Aut(M_{n'})$). From that, we can check that in fact, $S_m(M)\approx \varprojlim_n S_{m_n}(M_n)$ and $E(\Aut(M),S_m(M))\cong \varprojlim_n E(\Aut(M_n),S_{m_n}(M_n))$, and similarly to the case of products, the minimal left ideals and the Ellis groups in $E(\Aut(M),S_m(M))$ are the inverse limits of minimal left ideals and Ellis groups in $E(\Aut(M_n),S_{m_n}(M_n))$, respectively.
		In particular, the Ellis group $u\cM$ of $E(\Aut(M),S_m(M))$ is isomorphic to the profinite completion of integers $\widehat \bbZ=\varprojlim_n\bbZ/n\bbZ$. By analysis analogous to the preceding example, we see that $H(u\cM)$ and $D$ are trivial, and $\ker \hat r$ corresponds to the elements of $\widehat \bbZ$ represented by bounded sequences. Those sequences are exactly the elements of $\bbZ\subseteq \widehat \bbZ$ (this follows from the observation that a bounded sequence representing an element of $\widehat \bbZ$ has to eventually stabilize).
		
		Thus, by Theorem~\ref{thm:main}, $\Gal(T)$ is the quotient $\widehat \bbZ/\bbZ$ (which, again, has trivial topology), and, since the theory is NIP (because it is interpretable in an o-minimal theory), $\Gal(T)$ also has the Borel cardinality of $\widehat\bbZ/\bbZ$ which is $E_0$ (which can be seen as a consequence of the fact that it is hyperfinite (as an orbit equivalence relation of a $\bbZ$-action) and non-smooth (as the quotient of a compact Polish group by a non-closed subgroup)).
	\end{ex}
	
	\begin{rem}
		One can show that for $M_1$ (the pure circular order) the Ellis group of $(\Aut(M_1),S_{m_1}(M_1))$ is $\bbZ/1\bbZ$, i.e.\ trivial. However, Lemma~\ref{lem:ellis_group_onesort} does not hold for $n=1$, so one needs a different argument for this particular case (which we will not discuss here). In consequence, we could include $n=1$ in both examples given above (with the same conclusions).
	\end{rem}

	\printbibliography
\end{document}